\documentclass[letterpaper,11pt]{article}

\usepackage[english]{babel}
\usepackage[utf8]{inputenc}

\usepackage[
  margin=1.2in,
  top=1in,
  bottom=1.2in
]{geometry}

\makeatletter
\g@addto@macro\normalsize{%
  \setlength\abovedisplayskip{7pt}%
  \setlength\belowdisplayskip{7pt}%
  \setlength\abovedisplayshortskip{7pt}%
  \setlength\belowdisplayshortskip{7pt}%
}
\makeatother

\usepackage{graphicx}
\usepackage{cancel}
\usepackage{enumerate}
\usepackage{scalerel}
\usepackage{stackengine}
\usepackage{graphics}
\usepackage{fancyhdr}
\usepackage{amssymb}
\usepackage{amsmath}
\usepackage{mdwlist}
\usepackage{etoolbox}
\usepackage{latexsym}
\usepackage{amsthm}
\usepackage{aliascnt}
\usepackage{multicol}
\usepackage{makeidx}
\usepackage{bm}
\usepackage{dsfont}
\usepackage{wrapfig}
\usepackage{mdframed}
\usepackage[dvipsnames]{xcolor}
\usepackage{mathtools}
\usepackage{bbm}
\usepackage{tikz}

\usetikzlibrary{matrix}
\usetikzlibrary{positioning}

\usepackage{hyperref}
\usepackage[capitalise]{cleveref}
\usepackage{titlesec}

\titleformat{\section}
  {\large\centering\bfseries}
  {\thesection.}
  {.7em}
  {}

\titlespacing*{\section}
  {0pt}
  {3.5ex plus 0ex minus 0ex}
  {1.5ex plus 0ex}

\titleformat{\subsection}
  {\centering\bfseries}
  {\thesubsection.}
  {.7em}
  {}

\titlespacing*{\subsection}
  {0pt}
  {3.5ex plus 0ex minus 0ex}
  {1.5ex plus 0ex}

\titleformat{\subsubsection}
  {\centering\bfseries}
  {\thesubsubsection.}
  {.7em}
  {}

\titlespacing*{\subsubsection}
  {0pt}
  {3.5ex plus 0ex minus 0ex}
  {1.5ex plus 0ex}

\addto\captionsenglish{%
}

\usepackage{titling}
\newtheorem{theorem}{Theorem}[section]

\newaliascnt{corollary}{theorem}
\newtheorem{corollary}[corollary]{Corollary}
\aliascntresetthe{corollary}

\newaliascnt{conjecture}{theorem}

\aliascntresetthe{conjecture}

\newaliascnt{question}{theorem}

\aliascntresetthe{question}

\newaliascnt{lemma}{theorem}
\newtheorem{lemma}[lemma]{Lemma}
\aliascntresetthe{lemma}

\newaliascnt{proposition}{theorem}
\newtheorem{proposition}[proposition]{Proposition}
\aliascntresetthe{proposition}

\newaliascnt{maincorollary}{maintheorem}

\aliascntresetthe{maincorollary}

\newaliascnt{mainquestion}{maintheorem}

\aliascntresetthe{mainquestion}

\newtheoremstyle{definition}
  {2mm}
  {2mm}
  {}
  {}
  {\bfseries}
  {.}
  {.5em}
  {}

\theoremstyle{definition}

\newaliascnt{definition}{theorem}
\newtheorem{definition}[definition]{Definition}
\aliascntresetthe{definition}

\newaliascnt{remark}{theorem}
\newtheorem{remark}[remark]{Remark}
\aliascntresetthe{remark}

\newtheorem*{remark*}{Remark}

\newaliascnt{example}{theorem}

\aliascntresetthe{example}

\theoremstyle{plain}
\newtheorem*{namedthm}{\namedthmname}



\crefname{theorem}{Theorem}{Theorems}
\Crefname{theorem}{Theorem}{Theorems}

\crefname{proposition}{Proposition}{Propositions}
\Crefname{proposition}{Proposition}{Propositions}

\crefname{lemma}{Lemma}{Lemmas}
\Crefname{lemma}{Lemma}{Lemmas}

\crefname{corollary}{Corollary}{Corollaries}
\Crefname{corollary}{Corollary}{Corollaries}

\crefname{conjecture}{Conjecture}{Conjectures}
\Crefname{conjecture}{Conjecture}{Conjectures}

\crefname{question}{Question}{Questions}
\Crefname{question}{Question}{Questions}

\crefname{definition}{Definition}{Definitions}
\Crefname{definition}{Definition}{Definitions}

\crefname{remark}{Remark}{Remarks}
\Crefname{remark}{Remark}{Remarks}

\crefname{example}{Example}{Examples}
\Crefname{example}{Example}{Examples}

\crefname{maintheorem}{Theorem}{Theorems}
\Crefname{maintheorem}{Theorem}{Theorems}

\crefname{maincorollary}{Corollary}{Corollaries}
\Crefname{maincorollary}{Corollary}{Corollaries}

\crefname{mainquestion}{Question}{Questions}
\Crefname{mainquestion}{Question}{Questions}

\newcommand{\supp}{\operatorname{supp}}
\newcommand{\HP}{\operatorname{HP}}
\newcommand{\gen}{\operatorname{gen}}

\def\R{{\mathbb R}}
\def\Z{{\mathbb Z}}
\def\C{{\mathbb C}}

\def\N{{\mathbb N}}

\def\A{{\mathbb A}}
\def\P{{\mathbb P}}

\def\T{{\mathbb T}}

\def\cF{{\mathcal F}}

\def\cI{{\mathcal I}}

\def\cO{{\mathcal O}}

\def\cM{{\mathcal M}}

\def\Xmt{(X,\mu,T)}

\newcommand{\norm}[1]{\left\lVert#1\right\rVert}

\newcommand{\nilbohr}[1]{%
  \operatorname{Nil_{#1}-Bohr}%
}

\newcommand{\diff}{\mathop{}\!\mathrm{d}}

\newcommand{\hknorm}[1]{%
  {\left\vert\kern-0.25ex
   \left\vert\kern-0.25ex
   \left\vert #1
   \right\vert\kern-0.25ex
   \right\vert\kern-0.25ex
   \right\vert}%
}

\newcommand{\1}{\ensuremath{\mathds{1}}}

\newcommand{\E}{\mathbb{E}}

\newcommand{\ind}[1]{\mathbbm{1}_{#1}}

\usepackage[normalem]{ulem}

\makeatletter
\def\thanks#1{%
  \begingroup
  \let\@thefnmark\relax
  \@footnotetext{#1}%
  \endgroup
}
\makeatother

\begin{document}
\author{{Felipe Hernández}~~and~~{Tristán Radi\'c}}
\date{\small \today}
\title{{\bfseries Infinite prime sumsets in structured and $U^k(\Phi)$-uniform sets}}
\thanks{The first author acknowledges OpenAI’s ChatGPT for Academic Researchers program for providing access to ChatGPT during the preparation of this work. The second author was partially supported by a Simons Foundation International grant administered by the Simons Foundation SFI-MPS-SDF-00025709 TR and the National Science Foundation grant DMS-2348315.}

\maketitle

\begin{abstract}
By introducing new ergodic-theoretic techniques in nilsystems, we determine which infinite sumset configurations occur in $U^k(\Phi)$-uniform and $\nilbohr{}$ sets. To be more precise, our first result associates the degree $k$ of a $U^k(\Phi)$-uniform set  with the variety of sumsets it  contains, solving a conjecture of Kra, Moreira, Richter and Robertson. Restricting to $\nilbohr{}$ sets we show the existence of infinite sumsets with summands in the shifted primes $\P-1$. As a consequence, we show that for any real polynomial $Q(n)$ with leading irrational coefficient of degree $k$, and any natural numbers $\ell_1, \cdots, \ell_k$ there is an infinite set $P\subset \P$ such that 
\begin{equation*}
     Q\Big(\sum_{p \in I} p\Big)  \in  U \pmod 1  \quad \text{ for all } I \subset P
     , |I| = \ell_1, \ldots, \ell_k.  
\end{equation*}
\end{abstract}

\tableofcontents

\section{Introduction}


In 1947, Vinogradov \cite{vinogradov1947method} proved that for any real polynomial $Q(n)$ with irrational leading coefficient, the sequence $(Q(p))_{p \in \P}$ is equidistributed $ \pmod 1$, where $\P$ denotes the set of prime numbers. In particular, for any non-empty interval $U \subset \T = \R / \Z$, there is an infinite set of primes $P \subset \P$, such that $Q(p) \in U \pmod 1$ for all $p \in P$. More recently, answering a longstanding question of Erd\H{o}s \cite{Erdos77}, Kra, Moreira, Richter and Robertson \cite{kmrr25} proved that for any set of positive density $A \subset \N$ and any $k \in \N$, there exists an infinite set $B \subset \N$ and a shift $t \in \N$ such that
\begin{equation} \label{eq kmrr}
    \Big\{ \sum_{b \in I} b \colon I \subset B, 1\leq |I|  \leq k \Big\}   \subset A -t. 
\end{equation}
This constitutes a density version of Hindman's theorem \cite{Hindman74}, where by parity obstructions the shift $t \in \N$ is necessary and depends on $k$.

In this work we give a joint generalization of these theorems by showing that there is an infinite set $P \subset \P$ for which not only $Q(p) \in U \pmod 1$, but also all the possible sums with at most $k = \deg (Q(n))$ distinct elements from $P$. More generally we get:
\begin{theorem} \label{main motivating thrm}
    Let $U \subset \T$ be a non-empty interval and $Q(n) $ be a polynomial of degree $k \geq 1$ with leading irrational coefficient. For any distinct natural numbers $\ell_1, \ldots, \ell_k$, there is $P \subset \P$ infinite such that 
\begin{equation*}
     Q\Big(\sum_{p \in I} p\Big)  \in  U \pmod 1  \quad \text{ for all } I \subset P
     , |I| = \ell_1, \ldots, \ell_k.  
\end{equation*}
\end{theorem}

This theorem is a special case of \cref{main-theorem-2} concerning return times in nilsystems. Prior to proving the prime-restricted version in \cref{main motivating thrm}, we focus on the analogue where the summands range over $\mathbb{N}$. This is the content of \cref{main thrm uniform sets} that provides a positive answer to Conjecture 3.26 introduced by Kra, Moreira, Richter and Robertson in \cite{kra_Moreira_Richter_Roberson2025problems}. 

\cref{main thrm uniform sets} highlights classes of subsets of the natural numbers for which the shift $t \in \N$ in \eqref{eq kmrr} is not necessary. These classes are the $U^k(\Phi)$-uniform sets that we formally defined in \cref{sec local unif semi and f correspondence}. A $U^k(\Phi)$-uniform set is a set that avoids higher order parity obstruction arising from rotations and, more generally, from $(k-1)$-step nilsystems. 
As an example, the set $\{ n \in \N \colon Q(n) \in U \pmod 1\}$ from \cref{main motivating thrm} is $U^k(\Phi)$-uniform. For further examples of $U^k(\Phi)$-uniform sets, including sets coming from Hardy field functions and automatic sequences, we refer the reader to \cite[Sections 6 and 7]{radic2026Uniformity}. 

\cref{main thrm uniform sets} uncovers the relation between the degree $k$ of a $U^k (\Phi)$-uniform set, which measures its pseudo-random behavior, and the number of infinite sumsets it contains.

\begin{theorem} \label{main thrm uniform sets}
    Fix $k \geq 2$. Let $A_1, \ldots , A_k \subset \N$ be $U^k ( \Phi)$-uniform sets and  let  $\ell_1, \ell_2, \ldots, \ell_k$ be 
    distinct natural numbers. There exists $B \subset \N$ infinite such that
    \begin{equation}
        \Big\{ \sum_{b \in I} b \colon I \subset B, |I| = \ell_i \Big\}   \subset A_i \quad \text{ for all } i = 1, \ldots, k.
    \end{equation}
\end{theorem}

This result was partially solved in \cite{kra_Moreira_Richter_Roberson2025problems} for the case $k=2$ and $\ell_1 =1$, $\ell_2 = 2$ and subsequently generalized in 
\cite{radic2026Uniformity} for any $k \geq 2$ and $\ell_i = i $ for all $i=1, \ldots, k$. Also, again for $k=2$ and $\ell_1 =1$, $\ell_2 = 2$, a polynomial version of the result appeared in \cite[Corollary 11.4]{ackelsberg2025infinitepolynomial}. A related result was previously proved in \cite{diNasso_Golbring_Jin_Leth_Lupini_Mahlburg2015sumset}. A general version of \cref{main thrm uniform sets} was open even when considering a single $U^k(\Phi)$-uniform set $A$ instead of various $A_1, \ldots, A_k$.


When willing to restrict the set of summands $B$, there are further obstructions one needs to address. For instance, if a set $S\subset \N$ and a F\o lner sequence $\Phi$ are such that for all $A \subset \N$ with $\diff_\Phi(A) >0$ there is $B \subset S$ infinite and $t \in \N$ such that $\{b + b' \colon b,b' \in B, b \neq b'\} \subset A-t$, then $S$ has positive density (see \cite[Proposition 3.2]{kra_Moreira_Richter_Roberson2025problems}). In particular, even after shifting by a constant $t \in \N$, it is impossible to generalize \cref{main motivating thrm} by replacing the structured set $A=\{ n \in \N \colon Q(n) \in U \pmod 1\}$ by an arbitrary set with positive density. 

In our next result, we show that $\nilbohr{}$ sets provide a suitable family to restrict the sumset to summands in the shifted primes $\P-1$. $\nilbohr{}$ sets are non-empty subsets of the natural numbers that contain the return times of a point $x \in X$ in a nilsystem $(X,T)$ to a non-empty open set $U \subset X$, that is a superset of $\{n \in \N \colon T^n x \in U\}$. These sets were introduced in \cite{host_kra2011nil_Bohr} and extensively studied for their combinatorial richness (see  \cite{bergelson_Leibman2018ipr,Huang_Shao_Ye_nilbohr_automorphy:2016,shao_teravainem2020bombieri}).
\begin{theorem} \label{main thrm nilborh}
    Let $A \subset \N$ be a $\nilbohr{}$ set and $k \in \N$. There exists $t \in \N $  and $B \subset \P-1$ infinite such that
    \begin{equation}
        \Big\{ \sum_{b \in I} b \colon I \subset B, |I| = 1, \ldots, k \Big\} \subset A -t.
    \end{equation}
\end{theorem}
Notice that unlike \cref{main motivating thrm}, the summands in \cref{main thrm nilborh} are taken over $\P -1$ instead of $\P$ and that we also need a shift $t \in \N$. The result that allows us to deduce \cref{main motivating thrm} is somehow a combination of \cref{main thrm uniform sets} and \cref{main thrm nilborh} and it is stated in \cref{main-theorem-2}.

All the combinatorial results previously listed are proved in \cref{sec proofs} and are consequences of a better understanding of properties of nilsystems. This is the core of this work and we briefly describe it in the next section (for preliminary definitions see \cref{sec preliminaries}).

\subsection{New dynamic tools}

When studying arithmetic progressions in subsets of the natural numbers with positive density (i.e. Szemerédi's theorem \cite{Szemeredi75}) and its dynamic counterpart, this is, multiple ergodic averages (see Furstenberg \cite{Furstenberg77}), Host and Kra proved that nilsystems play a central role  by showing that inverse limit of nilsystems are {\em characteristic factors} of multiple ergodic averages \cite{Host_Kra_nonconventional_averages_nilmanifolds:2005}. Similarly, Kra, Moreira, Richter, and Robertson proved that in order to capture infinite sumset, after applying the Furstenberg correspondence principle, it is enough to understand measures constructed from nilsystems. With this, many of the tools developed to understand arithmetic progressions were successfully adapted to the study of infinite sumsets. This was systematically carried over in \cite{kmrr1,kmrr_BB,kmrr25} and generalized in various directions, notably \cite{ackelsberg_jamneshan2025equidistribution,DATI2025abelianBB,hernandez2025infinite,hernandez_kousek_radic2025density_Hindman,kousek_radic2025unrestrictedBB,kousek2025asymmetric}.

We separate our main contribution into two levels: conceptual and technical. The conceptual one consists of further exploiting tools that arise from the study of arithmetic progressions in nilsystems. In particular, we study the \emph{Hall-Petresco nilmanifold} and prove that its Haar measure is \emph{progressive} (see definitions below). The technical result is a new phenomenon in nilsystems that we call a \emph{fiberwise continuous ergodic decomposition} which we introduce to prove the progressive property. 

Expanding on the first point, in a non-abelian nilpotent Lie group $G$, the set of arithmetic progression $(g, gh, gh^2, \ldots, gh^\ell)$ for $g,h \in G$ is not a subgroup of $G^{\ell +1}$. The Hall-Petresco group $\HP_\ell(G) \subset G^{\ell+1}$ (named after the work of Hall \cite{hall1934HP} and Petresco \cite{petresco1954HP}) is the group generated by such arithmetic progressions. This subgroup has an algebraic description and its modern understanding has been notably developed by the work of Leibman \cite{leibman1998HP} (see also \cite[Chapter 15]{Host_Kra_nilpotent_structures_ergodic_theory:2018}).

For an ergodic nilsystem $(Z, m,T)$ with $Z=G/\Gamma$ and a natural number $\ell \geq 2$, the Hall-Petresco nilmanifold $\HP_\ell (Z)$ is defined as $\{ u (e_Z, \ldots, e_Z) \in Z^{\ell+1} \colon u \in \HP_\ell (G) \}$ with Haar measure $m_{\HP_\ell(Z)}$. Also, for a given $z \in Z$, we denote $\rho_z$ the Haar measure of the subnilmanifold of $\HP_\ell (Z)$ with fixed first coordinate  $  \{ w \in \HP_\ell(Z) \colon w_0 = z \} $. It was proved in \cite{ziegler2005nonconventional} (see also \cite{Bergelson_Host_Kra05}) that for every $f_0, f_1, \ldots, f_\ell \in C(X)$, every F\o lner sequence $\Phi$ and for $m$-almost every $z \in Z$, one has:
    \begin{equation}\label{theorem HP ergodic averages}
        \lim_{N \to \infty} \frac{1}{|\Phi_N|} \sum_{n \in \Phi_N} f_0(z) f_1(T^n z) f_2(T^{2n}z) \cdots f_\ell (T^{\ell n } z) = \int_{\HP_\ell(Z)} f_0 \otimes f_1 \otimes \cdots \otimes f_\ell \diff \rho_z. 
    \end{equation}
This illustrates that the measure $\rho_z$ naturally arises when studying multiple ergodic averages in nilsystems and therefore in general systems. Our main contribution is to prove that $\rho_z$ is progressive for all $z \in Z$, see \cref{thrm rho is unif progressive} (this concept was introduced in \cite{kmrr25} and we recall it in \cref{progressive measures}). For the discussion, any progressive measure $\tau \in \cM(Z^{\ell+1})$ with $\tau(\{z\} \times Z^\ell)=1$ fulfills that if $V_1, \ldots, V_\ell \subset Z$ are open sets with $\tau(X \times V_1 \times \cdots \times V_\ell) >0$ then there is an infinite $B \subset \N$ such that 
\begin{equation} \label{eq progressive measures intro}
    \Big\{ \sum_{b \in I} b \colon I \subset B, |I|=i\Big\} \subset \{ n \in \N \colon T^n z \in V_i \} \quad \text{ for all } i=1, \ldots, \ell. 
\end{equation}
Therefore, progressive measures are the fundamental tool for deriving infinite sumset properties from dynamics.

Notice that by properties of nilsystems, for all $z \in Z$, the averages on the left-hand side of \eqref{theorem HP ergodic averages} always define a measure that we denote $\sigma_z$. In other words $\sigma_z = \rho_z$ for $m$-almost every $z \in Z$, but whenever $s \geq 2$, the set of exceptions for this equality is non-empty in $s$-step nilsystems. From \cite{kmrr25} we know that $\sigma_z$ is progressive for all $z \in Z$ and therefore $\rho_z$ is progressive for almost every $z \in Z$. 
However, having the result for $\rho_z$ in a set of full measure is not enough to deduce the combinatorial results as it is illustrated in \cite[Section 8]{radic2026Uniformity}. Upgrading the result from \emph{almost every} to \emph{every} forces us to develop new technical results, and thus to obtain a better understanding of the measure $\rho_z$.    

Before explaining the technical results, the main advantage and key feature of the measure $\rho_z \in \cM(\HP_\ell (Z))$ is the \emph{good projection property} (which is not true for general $\sigma_z$, see \cite[Appendix A]{hernandez_kousek_radic2025density_Hindman}). The good projection property is explained in \cref{sec good projections}, but roughly speaking, assuming that $Z$ is an $(\ell-1)$-step nilsystem, it ensures that if one fixes $k \leq \ell$ coordinates and considers the measure resulting after projecting $\rho_z$ on those coordinates, then that measure depends only on the $(k-1)$-step pronilfactor $Z_{k-1}$. 
This allows us to handle an expression that a priori involves an $(\ell-1)$-step nilsystem and reduce it to an $(k-1)$-step nilsystem where we can use fundamental properties of $U^k(\Phi)$-uniform sets to conclude  \cref{main thrm uniform sets}.

To prove that $\rho_z$ is progressive for all $z \in Z$, we use that $z \mapsto \rho_z$ is the continuous ergodic decomposition of the Haar measure $m_{\HP_\ell (Z)}$, see \cref{sec cont ergodic decomp} for definitions. In fact, the measure $m_{\HP_\ell (Z)}$ is invariant under the joint action of $T_\Delta = T \times \cdots \times T$ and $\tilde T = Id \times T \times \cdots \times T^\ell$, so it has two distinct ergodic decompositions (one for each transformation).
The key new technical result, that is of independent interest, is the \emph{fiberwise continuous ergodic decomposition} (see \cref{theorem cross ergodic decompositions}) which ensures that for a nilsystem $(Y,m_Y,S,R)$ if $y \mapsto\mu_y$ and $y \mapsto\nu_y$ are continuous $S$-ergodic and respectively $R$-ergodic decompositions, then for every $y \in Y$ the set of $R$-generic points $z \in Y$ for $\nu_z$ has full $\mu_y$ measure. This generalizes \cite[Lemma 6.7]{kmrr1} where the existence of a continuous ergodic decomposition in nilsystem was established.
We derive this result using new properties of non-ergodic nilsystems obtained by the first author in \cite{hernandez2026nonergodicNil}, which we further develop in this paper.

After establishing that $\rho_z$ is progressive, using Gowers-uniformity properties of the von Mangoldt function $\Lambda(n)$ developed by Green-Tao \cite{Green_Tao10} and Green-Tao and Ziegler \cite{Green_Tao_Ziegler12}, we upgrade the result and prove that $\rho_z$ is $(\P-1)$-progressive. In other words, we prove that the set $B$ involved in \eqref{eq progressive measures intro} can be forced to take values in $\P-1$. 

Unlike the previous case, for shifted primes we cannot reduce the study of general systems to nilsystems, as it is no longer possible to exploit the flexibility of choosing different F\o lner sequences.  At a technical level, since the set shifted primes $\P - 1$ has zero density, we know that this reduction is impossible (see previous discussion). However, the results hold at the nilsystem level allowing us to deduce \cref{main thrm nilborh} for $\nilbohr{}$ sets.

To finish this section, we now state some of the consequences of our results in topological dynamics that are of independent interest. This extends second author's \cite[Theorem 4.6]{radic2026Uniformity} and for the full statement refer to \cref{theorem summary}. Here we highlight two equivalent conditions, the first one is a finite configuration related to the cube completion property in nilsystems exploited by Host, Kra and Maass in \cite{Host_Kra_Maass_nilstructure:2010} and the second one is a new infinite configuration arising from shifted primes. 
\begin{theorem} \label{theorem summary intro} 
    Let $k \geq 2$ be an integer, $(X,T)$  be a minimal nilsystem. For $x,y \in X$, the following are equivalent
    \begin{enumerate}
    \item For every neighborhood $V$ of $y$ there exists $b_1, \ldots, b_k \in \N$ distinct natural numbers such that 
        $$ \Big\{ \sum_{i \in J} b_i \colon   J \subset \{1,\ldots, k\}, \ J \neq \emptyset\} \subset \{ n \in \N \colon T^n x \in V \};$$
 
        \item   For every neighborhood $V$ of $y$ and distinct natural numbers $\ell_1,\cdots, \ell_k$, there exists an infinite set $B \subset \P-1$ such that 
    $$ \Big\{ \sum_{b \in I} b \colon I \subset B, |I|=\ell_1, \ldots, \ell_k\Big\} \subset \{ n \in \N \colon T^n x \in V \}.$$
\end{enumerate}
    
\end{theorem}

\section*{Acknowledgments}

The authors thank Bryna Kra for many useful comments on a former draft of this document. The authors also thank Axel Álvarez, whose conversations helped sharpen some results in \cref{sec top consequences}. 

\noindent \textbf{Use of technology:} No AI tools were used to produce any of the results. ChatGPT use was limited to proof checking and correcting grammar.

\section{Preliminaries} \label{sec preliminaries}
We follow the notation and framework of the earlier work of the second author in \cite{radic2026Uniformity}. For completeness, we recall the needed definitions, as well as introducing the new tools related to Gowers norms that we  need for the results related to primes.
 \subsection{Dynamical systems}
 A \emph{topological dynamical system} is a pair $(X,T)$ where $X$ is a compact metric space and $T:X\to X$ is a homeomorphism. A topological dynamical system $(X,T)$ is minimal if for every point $a\in X$ we have that its orbit $\overline{O_T(a)}= \overline{\{T^n a | n\in \N\} }$ is equal to $X$. A \emph{measure preserving system} (or system for short) is a triple $(X,\mu,T)$ if $(X,T)$ is a topological dynamical system and $\mu$ is a $T$-invariant Borel probability measure. We say that $(X,\mu,T)$ is \emph{ergodic} if for every measurable set $A\subseteq X$, $\mu(A\Delta T^{-1}A)=0$ implies $\mu(A)=0$ or $1$. We say that $T$ is \emph{totally ergodic} if $(X,\mu,T^k)$ is ergodic for every $k\in \N$. The definitions for $\Z^d$-actions are analogous, so we skip them and refer the reader to \cite{Glasner_ergodic_theory_joinings:2003} for further details.

A \emph{F\o lner sequence} $\Phi=(\Phi_N)_{N\in \N}$ is a sequence of finite subsets of the natural numbers for which  $|\Phi_N\Delta (\Phi_N+t)|/|\Phi_N|$ tends to zero as $N\to \infty$ for every $t\in \N$. 
For a system $(X,\mu,T)$, we say that a point $a\in X$ is \emph{generic along a F\o lner sequence $\Phi$}, denoted as $a\in \gen(\mu,\Phi)$, if for all continuous functions $f\in C(X)$
$$ \lim_{N\to \infty} \frac{1}{|\Phi_N|} \sum_{n\in \Phi_N} f(T^na) = \int_X f d\mu.  $$
The set of generic points has full measure when $\mu$ is ergodic. 

For a Borel map $\theta:X\to Y$ between topological spaces $X$ and $Y$, if $\mu$ is a Borel probability measure on $X$, then the pushforward measure $\nu$ on $Y$, denoted by $\nu=\theta_*\mu$, is given by $\nu(A)=\mu(\theta^{-1}(A))$ for all measurable sets $A\subseteq Y$. Given two measure preserving systems $(X,\mu,T)$ and $(Y,\nu,S)$, we say that $(Y,\nu,S)$ is a \emph{factor} of $(X,\mu,T)$ (or analogously that $(X,\mu,T)$ is a extension of $(Y,\nu,S)$) if there exists a measurable map $\pi:X\to Y$, called a \emph{factor map}, such that $\pi_*\mu=\nu$ and $\pi\circ T= S\circ \pi $, up to null-sets. If the factor map $\pi:X\to Y$ is continuous and surjective, we say that $\pi$ is a \emph{continuous factor map} and $Y$ is a \emph{topological factor} of $X$. We  abuse notation by using the same letter $T$ for the transformation in $X$ and in its factor or extension $Y$. For a factor map $\pi:X\to Y$ and a function $f\in L^2(\mu)$, we denote by $\E(f|Y)$ the conditional expectation $\E(f | \pi^{-1}(\mathcal{B}(Y)))$, where $\mathcal{B}(Y)$ is the Borel $\sigma$-algebra on $Y$. Since $\mu$-a.e. $\E(f|Y)(x)$ depends only in its image $\pi(x)$ in $Y$, we make the abuse of notation of also writing $\E(f|Y) (y)$ and thinking of $\E(f|Y)$ as a function in $L^2(\nu)$.

To finish this subsection, we define the notion of \emph{inverse limit}. For a sequence $(X_j,\mu_j,T)_{j\in \N}$ of systems for which there is a factor map $\theta_j:X_{j+1}\to X_j$ for each $j\in \N$, we say that $(X,\mu,T)$ is the inverse limit of this sequence if for each $j\in \N$ there exists a factor map $\pi_j:X\to X_j$ satisfying that $\pi_{j}=\theta_j\circ \pi_{j+1}$ for every $j\in \N$ and that
$$ \bigcup_{j\in \N} \{ f\circ \pi_j : f\in L^2(\mu_j)\}  $$
is dense in $L^2(\mu)$. We also define the analogous notion for topological dynamical system, by replacing factor maps by continuous factor maps, and $f\in L^2(\mu_j)$ by $f\in C(X_j)$.

\subsection{Nilmanifolds, nilsystems and Host-Kra norms}\label{preliminaries-2}
Let $G$ be a $k$-step nilpotent Lie group and $\Gamma$ a uniform lattice in $G$ (this is, a discrete cocompact subgroup). We call the compact manifold $X=G/\Gamma$ a \emph{$k$-step nilmanifold}, or just \emph{nilmanifold} if the step is understood by context. We endow $X=G/\Gamma$ with its Haar measure $m_X$, this is, its unique left-translation invariant Borel probability measure. A subgroup $H\subseteq G$ is call \emph{$\Gamma$-rational}, or simply \emph{rational}, if $H\Gamma$ is closed in $G$. In particular, rational groups are closed (see \cite[Lemma 14, Chapter 10]{Host_Kra_nilpotent_structures_ergodic_theory:2018}). We denote by $G^\circ$ the connected component of the identity $e_G$ in $G$. Then, $G^\circ$ is an open normal subgroup of $G$. We  assume without loss of generality that $G^\circ$ is simply connected (see \cite[Lemma 7, Chapter 10]{Host_Kra_nilpotent_structures_ergodic_theory:2018})

Consider $X=G/\Gamma$ a nilmanifold. Let $T:X\to X$ be the action induced by left translation of a fixed element $\tau \in G$. We say that $(X,m_X,T)$ is a $k$-step nilsystem. We generalize this to $\Z^d$-actions in the natural way: $(X,\mu,T_1,\ldots,T_d)$  is a $k$-step $\Z^d$-nilsystem where $\tau_1,\ldots,\tau_d\in G$ commute and $T_i:X\to X$ are their respective induced maps by left translation. As it is customary (see for example \cite[Chapter 10, Lemma 7 and Theorem 12]{Host_Kra_nilpotent_structures_ergodic_theory:2018}), we  assume without loss of generality that $G$ is generated by $G^\circ$ and the elements $\tau_1,\ldots,\tau_d$. 

An inverse limit of $k$-step nilsystems is called a \emph{$k$-step pronilsystem}. An \emph{$\infty$-step pronilsystem} is an inverse limit of nilsystem (i.e. with no restriction in the step). We say that a system is a pronilsystem if it is a $k$-step pronilsystems for some $k\in \N\cup \{\infty\}$.

Properties like minimality, ergodicity and unique ergodicity are equivalent in pronilsystems, see for example \cite{Host_Kra_nilpotent_structures_ergodic_theory:2018}. Moreover, if $(X,T_1,\ldots,T_d)$ is a pronilsystem, then for every $x\in X$ the orbit $$\overline{O_{T_1,\ldots,T_d}(x)}=\overline{  \{ T_1^{n_1} \cdots T_d^{n_d} x | (n_1,\ldots,n_d)\in \Z^d \} }$$
is minimal and uniquely ergodic by \cite{Leibman_pointwise_conv_polynomial_nil:2005}. 

Let $(X, \mu, T)$ be an ergodic measure preserving system. For every $k \in \mathbb{N} \cup\{\infty\}$, $(X, \mu, T)$ admits a maximal factor isomorphic to a $k$-step pronilsystem, called the \emph{$k$-step pronilfactor}, denoted by $Z_k(X)$ or just $Z_k$ depending on the context. In what follows, we always consider the topological model of $Z_k$ which is a $k$-step pronilsystem. For $k=1$, this factor is also known as the Kronecker factor and can be characterized as the factor spanned by the eigenfunctions of $(X, \mu, T)$, that is the functions $f \in L^2(\mu)$ such that $T f=\lambda f$ for some $\lambda \in \mathbb{S}^1$. An ergodic system $(X, \mu, T)$ is weak mixing if $\lambda=1$ is the unique eigenvalue. If $(X, \mu, T)$ is weak mixing then $Z_k(X)$ is the trivial factor for all $k \in \mathbb{N}$. Similarly, if for some $k \in \mathbb{N}, Z_{k+1}=Z_k$, then $Z_s=Z_k$ for all $s \geq k$.

In \cite{Host_Kra_nonconventional_averages_nilmanifolds:2005}, Host and Kra showed that these factors can be characterized by seminorms that we refer as Host–Kra seminorms.

\begin{definition}
    Let $(X, \mu,T)$ be an ergodic system and $f \in L^{\infty}(\mu)$. The $k$-uniformity seminorms of $f$, $\norm{f}_{U^{k}(X,\mu,T)}$, are defined inductively as follows: 
    \begin{align*}
        \norm{f}_{U^0(X, \mu,T)} &=  \int f d\mu \text{ and } \\
                \norm{ f}_{U^{k+1}(X, \mu,T)}^{2^{k+1}} &= \lim_{H \to \infty }\frac{1}{H} \sum_{h \leq H} \norm{f \cdot \overline{T^hf}}_{U^k(X, \mu,T)}^{2^k}.
    \end{align*}
\end{definition}
    The function $\norm{ \cdot }_{U^{k}(X, \mu,T)}$ defines a seminorm when $k \geq 1$. The main theorem in \cite{Host_Kra_nonconventional_averages_nilmanifolds:2005} shows that the $k$-step pronilfactor, $Z_{k}$, of an ergodic system $(X, \mu,T)$ is characterized by the property 
    \begin{equation*} 
        \norm{f}_{U^{k+1}(X, \mu,T)} = 0 \iff \E(f \mid Z_{k}) =0,
    \end{equation*}
    for any $f\in L^{\infty}(\mu)$.  We say that $(X,\mu,T)$ has \emph{topological pronilfactors} if the factor map $\pi_k:X\to Z_k$ is a topological factor map for every $k\in \N$.

We now define our notions of uniform and Nil-Bohr sets.
\begin{definition}\label{uniform-sets}
    We say that a measurable set $E\subseteq X$ is \emph{$U^k(X,\mu,T)$-uniform} for some $k\in \N$, if $\mu(E)>0$ and $\norm{\ind{E}-\mu(E)}_{U^k(X,\mu,T)}=0$. This is equivalent to $\E(\ind{E}| Z_{k-1})(x)=\mu(E)$ for $\mu$-a.e. $x\in X$.  
\end{definition}
 
\begin{definition} \label{def nilbohr}
 A set $D \subset \mathbb{N}$ is a $\nilbohr{s}_0$ set for $s \in \mathbb{N}$ if there exists an $s$-step nilsystem $(Z, T)$, a point $z_0 \in Z$ and an open neighborhood $U \subset X$ of $z_0$ such that $$D\supset \left\{n \in \mathbb{N} \mid T^n z_0 \in U\right\}.$$ 
 We say that $D\subseteq \N$ is a $\nilbohr{s}$ set if there a $\nilbohr{s}_0$ set $S$ and $t\in \N$ such that $S-t= D$.  
  We say that $D \subset \mathbb{N}$ is a $\nilbohr{}_0$ set (resp. $\nilbohr{}$ set) if it is a $\nilbohr{s}_0$ set (resp. $\nilbohr{s}$ set) for some $s \in \mathbb{N}$.
\end{definition}

\subsection{Local uniformity seminorms and Furstenberg correspondence principle} \label{sec local unif semi and f correspondence}

  A bounded sequence $\phi \colon \N \to \C$ \emph{admits correlations along the F\o lner sequence} $\Phi$ if for all $h_1, . . . , h_k \in \N$ and $\epsilon_1, \ldots, \epsilon_k \in \{0,1\}$, the limit
\begin{equation*}
    \lim_{N \to \infty} \frac{1}{|\Phi_N|} \sum_{n \in \Phi_N} \prod^k_{i=1} C^{\epsilon_i} \phi(n + h_i)
\end{equation*}
exists, where for a complex number $z$, $C z$ denotes the complex conjugation. For any bounded sequence $\phi\colon \N \to \C$ and any F\o lner sequence $\Phi$, there is a sub-F\o lner sequence $\Psi$ such that $\phi$ admits correlations along $\Psi$.  

Assume that a sequence $\phi \colon \N \to \C$ admits correlations along a given F\o lner sequence $\Phi$ we define, inductively in $k$, the \emph{local uniformity seminorms of $\phi$ along $\Phi$} by
\begin{align}
    &\norm{\phi}_{U^0(\Phi)} = \lim_{N \to \infty} \frac{1}{|\Phi_N|} \sum_{n \in \Phi_N} \phi(n) \nonumber \\
    &\norm{\phi}_{U^{k+1}(\Phi)}^{2^{k+1}} = \lim_{H \to \infty}\frac{1}{|H|} \sum_{h =1}^H \norm{\Delta_h \phi }^{2^k}_{U^k(\Phi)}  \label{eq local uniformity norms} 
\end{align}
where $(\Delta_h \phi)(n) = \phi(n) \overline{\phi(n+h)}$ for all $n,h \in \N$. This seminorms were introduced in \cite{Host_Kra09} and the existence of the limits in \eqref{eq local uniformity norms} was proved in that paper too. Similarly to \cref{uniform-sets}, a set $A \subset \N$ is $U^k(\Phi)$\emph{-uniform} if the indicator function $\1_A \colon \N \to \C$ admits correlations along $\Phi$, $\diff_{\Phi}(A) > 0$ and $\norm{\1_A - \diff_{\Phi}(A) }_{U^k (\Phi)} =0$.

   The local uniformity seminorm share many similarities with the Host-Kra seminorms. For instance, for $k \in \N$, $f \in L^\infty(\mu)$ in a measure preserving system $\Xmt$ and $\phi \in \ell^\infty (\N)$ admitting correlations along $\Phi$
  \begin{equation*}
      \norm{f}_{U^k\Xmt} \leq \norm{f}_{U^{k+1}\Xmt} \quad \text{ and } \quad  \norm{\phi}_{U^k(\Phi)} \leq \norm{\phi}_{U^{k+1}(\Phi)}. 
  \end{equation*}
  This similarities are directly linked with the Furstenberg correspondence principle. This was originally proved by Furstenberg in \cite{Furstenberg77} and here we use a version tailored for our purposes.

\begin{proposition}[{\cite[Corollary 2.11]{radic2026Uniformity}}] \label{furstenberg correspondence for sets}
    Let $k \in \N$ and let $A_1, \ldots, A_r \subset \N$ be $U^k(\Phi)$-uniform sets for $i =1, \ldots,r$. There exist  an ergodic system $(X, \mu,T) $ with topological pronilfactors, a F\o lner sequence $\Psi = (\Psi_N)_{N \in \N}$, a point $a \in \gen(\mu, \Psi)$ and $U^k(X, \mu,T)$-uniform clopen sets $E_i$ such that $\mu(E_i) = \diff_\Phi (A_i) >0$ and
    \begin{equation} \label{eq correspondence sequences}
        T^n a \in E_i \iff n \in A_i 
    \end{equation}
    for all $n \in \N$, $i=1, \ldots,r$.
\end{proposition}

We clarify that the statement in \cite[Corollary 2.11]{radic2026Uniformity} does not include the fact that the system $(X, \mu,T) $ has topological pronilfactors, but this extra assumption can be incorporated after \cite[Lemma 5.8]{kmrr1}.

\section{Ergodic decompositions in nilsystems} \label{sec ergodic decomposition}

\begin{definition} \label{defi cont desintegration}
  Let $(X, \mu,T)$ be a system and $(Y, \nu, S)$ be a factor with factor map $\pi\colon X \to Y$. We say that the map $y \mapsto \mu_y$ from $Y \to \cM(X)$ is a disintegration of $\mu$ over $\nu$ if 
  \begin{itemize}
      \item $y \mapsto \mu_y (E)$ is measurable for every Borel $E \subset X$
      \item $\mu_y(\pi^{-1}(y)) =1$ for $\nu$-almost every $y \in Y$,
      \item For every $f \in L^\infty(\mu)$
      \begin{equation*}
          \int_X f \diff \mu = \int_Y \int_X f \diff \mu_y \diff \nu(y).
      \end{equation*}
  \end{itemize}
\end{definition}
 As a consequence we have that $\int f \diff \mu_y = \E(f\mid Y) (y)$ for $\nu$-almost every $y \in Y$. Using the factor map we abuse notation and say that the map $x \mapsto \mu_{\pi(x)}$ for $x \in X$ is the disintegration of $\mu$ over $\nu$. Of special interest is when $Y$ is the factor given by the invariant $\sigma$-algebra $\cI(T)$, in that case we just write $x \mapsto \mu_x$ and we call it \emph{ergodic decomposition}. For an ergodic decomposition $x \mapsto \mu_x$, the measure $\mu_x$ is ergodic $\mu$-almost surely for $x \in X$.

 In \cite{kmrr1} and \cite{radic2026Uniformity} the authors applied a continuous version of these notions. In the case of the continuous ergodic decomposition if $x \mapsto \mu_x$ is continuous then we say that $\Xmt$ has \emph{continuous ergodic decomposition}. Similarly, for two dynamical system $\Xmt$ and $(Y, \nu,S)$, if the factor map $\pi \colon X \to Y$ is continuous and $y \mapsto \mu_y $ is continuous, then we call it a \emph{continuous disintegration over} $\nu$. 

\subsection{Continuous ergodic decomposition in nilsystems}  \label{sec cont ergodic decomp}

A key element in \cite{kmrr1} is the existence of a continuous ergodic decomposition in nilsystem. In this paper, we use a more detailed version of this theorem using recent results developed by the first author in \cite{hernandez2026nonergodicNil}.   

Fixing notation, in this section any nilpotent Lie group $G$ is generated by $G^\circ$ the connected component of the identity, and $d+1$ commuting elements $\tau, \sigma_1, \ldots, \sigma_d \in G$ (we later use it for two commuting elements $\tau$ and $\sigma$). We denote $\pi_\Gamma \colon G \to G / \Gamma =X$ the natural quotient map and $e_X = \pi_\Gamma(e_G)$. $(X, \mu, T, S_1, \ldots, S_d)$ is a nilsystem, where $T, S_1, \ldots, S_d \colon X \to X$ are nilrotation given by $Tx = \tau \cdot x$ and $S_i x = \sigma_i \cdot x$ for all $x \in X$, $i =1, \ldots, d$ respectively.

Using the convention of \cite{Host_Kra_nilpotent_structures_ergodic_theory:2018}, we say that $Y$ is a subnilmanifold of $X$ if there is a closed subgroup $H$ of $G$ and a point $y \in X$ such that $Y= Hy$. If $Y$ is also closed, we say that $H$ is a rational subgroup.

\begin{proposition} \label{Prop G_tau}
    Let  $G_\tau = \langle G^\circ , \tau\rangle $, then $G_\tau$ is normal rational group in $G$. Moreover, if $X_\tau = G_\tau e_X \subset X$ then there is a finite set $D \subset \Z^d$ such that 
    \begin{equation} \label{eq partition Xtau}
        X = \bigsqcup_{h \in D} S^{h_1}_1 \cdots S^{h_d}_d X_\tau. 
    \end{equation}
    In particular, \eqref{eq partition Xtau} holds true for any translate of $D$ in $\Z^d$.
\end{proposition}

\begin{proof}
    For the normality, since $G = \langle G^\circ , \tau, \sigma_1, \ldots, \sigma_d\rangle$ it is enough to prove that $\sigma_i G_\tau \sigma_i^{-1} = G_\tau$ for all $i =1, \ldots, d$. Fix $i\in [d]$. Using that $G^\circ$ is normal in $G$ and that $\tau$ commutes with $\sigma_i$ we get that
    \begin{align*}
        \sigma_i G_\tau \sigma_i^{-1} = \sigma_i \langle \tau, G^\circ \rangle \sigma_i^{-1} =  \langle \sigma_i \tau \sigma_i^{-1}, \sigma_i G^\circ \sigma_i^{-1} \rangle = \langle \tau, G^\circ \rangle = G_\tau.
    \end{align*}
    For the rationality, it is enough to notice
    $X_\tau$ is closed. Indeed, we  actually prove that if $X_0=G^\circ\Gamma$, then
\begin{equation}\label{eq-descriving-X-tau}
        X_\tau= \bigsqcup_{i=0}^r \tau^i X_0
    \end{equation}
    and since $X_0$ is closed (see \cite[Chapter 10]{Host_Kra_nilpotent_structures_ergodic_theory:2018}), then so is $X_\tau$. For doing this, since $X$ has finitely many isomorphic and disjoint connected components (see \cite[Section 1.1]{LEIBMAN_2007}) we have that $\mu(X_0)>0$. Since $\mu$ is $T$-invariant, by pigeonhole principle there is $d\in \N$ such that $T^{d+1}X_0\cap X_0\neq \emptyset$, and since $T$ maps connected components to connected components we get that $T^{d+1}X_0=X_0$. With this we deduce \eqref{eq-descriving-X-tau}  for $r$ the minimum such that $\tau^{r+1}X_0=X_0$. 
    
    
    Finally, $G_\tau \backslash G  $ is a discrete countable abelian group, so by compactness of $G/ \Gamma$, $G_\tau \backslash G /\Gamma$ is equivalent to a finite abelian group $\Z^d / \Lambda$ for some subgroup $\Lambda \leq \Z^d$.  In particular, $\sigma^{h_1}_1 \cdots \sigma^{h_d}_d G_\tau$ with $h \in \Z^d$, induces a finite partition of the form $S^{h_1}_1 \cdots S^{h_d}_d X_\tau$ for $h \in D$ where $D$ is a fundamental domain in $\Z^d$ of the finite group $\Z^d / \Lambda$.
\end{proof}

\begin{theorem}\label{Continuous-ergodic-decomp-with-Leibman-group}
    Let $(X = G/\Gamma, \mu, T, S_1, \ldots, S_d)$ be a nilsystem. Then there exists a continuous $T$-ergodic decomposition $x \mapsto \mu_x$ of $\mu$. Moreover, there is $H$ a rational subgroup which is normal in $G_\tau$ such that
    \begin{itemize}
        \item If $x \in X_\tau$, $\mu_x = m_{Hx}$ the Haar measure of the nilmanifold $Hx$.
        \item More generally, for $D$ as in \cref{Prop G_tau} with $0 \in D$, if $x = S_1^{h_1} \cdots S_d^{h_d} y$ for $y \in X_\tau$ and $h \in D$ then $\mu_x = S_1^{h_1} \cdots S_d^{h_d} \mu_y = S_1^{h_1} \cdots S_d^{h_d} m_{Hy}$. 
    \end{itemize}
\end{theorem}
Following the notation from \cite{hernandez2026nonergodicNil}, we call $H$ the \emph{Leibman group.} 
\begin{remark*}
    The finite set $D \subset \Z^d$ given by the \cref{Prop G_tau} is not unique. However, the previous characterization does not depend on the chosen finite set $D \subset \Z^d$, because the continuous ergodic decomposition is unique. 
\end{remark*}

\begin{proof}
    Consider $X_\tau = G_\tau e_X$ with Haar measure $m_{X_\tau}$, then by \cite[Proposition 3.6]{hernandez2026nonergodicNil}, there exists $H \leq G_\tau$ a rational subgroup such that for $m_{X_\tau}$-almost every $x \in X_\tau$, $\overline{O_T(x)} = Hx$. In particular, if we define $\mu_x=m_{Hx}$ for $x \in X_\tau$, then $\mu_x$ is ergodic $m_{X_\tau}$-almost surely. Moreover, $H$ can be taken to be generated by $\tau$ and $H^\circ$ the connected component of $H$ containing $e_G$, in particular $H$ is a closed subgroup in $G_\tau$. With this $Hx$ is a $T$-invariant subnilmanifold and $T \mu_x = \mu_{x}$ for all $x \in X_\tau$.  By the same proposition, $H$ is normal in $G_\tau$ (the original statement in \cite[Proposition 3.6]{hernandez2026nonergodicNil} assumes $G = G_\tau$, so adapting the proof to our context we get this version). 

    By \cite[Chapter 10, Proposition 16]{Host_Kra_nilpotent_structures_ergodic_theory:2018}, the map $x \mapsto m_{Hx}$ from $X_\tau$ to $\cM(X_\tau) \subset \cM(X)$ is continuous for the weak$^*$ topology and $\int_{X_\tau} \mu_x \diff m_{X_\tau}(x) = m_{X_\tau}$. Therefore, using the previous paragraph, $x \mapsto \mu_x$ defines a continuous ergodic decomposition of $m_{X_\tau}$. 

    Fix $D \subset \Z^d$ as in \cref{Prop G_tau} such that $0 \in D$, then 
    \begin{equation} \label{eq disjoint part measure}
        \mu = \frac{1}{|D|} \sum_{h \in D} S_1^{h_1} \cdots S_d^{h_d} m_{X_\tau}. 
    \end{equation}
    Fix $h \in D$, since $S_1, \ldots, S_d \colon X \to X$ are homeomorphism, for every $x \in S_1^{h_1} \cdots S_d^{h_d} X_\tau $ there is a unique $y \in X_\tau$ such that $S_1^{h_1} \cdots S_d^{h_d} y = x$. We denote such element $\psi_h(x) = y$, where $\psi_h$ is simply the continuous inverse map of $S_1^{h_1} \cdots S_d^{h_d}$.  In particular, the map $x \mapsto S_1^{h_1} \cdots S_d^{h_d} \mu_{\psi_h (x)}$ from $S_1^{h_1} \cdots S_d^{h_d} X_\tau$ to $\cM(X)$ is continuous. Also, by commutativity of $T$ with $S_i$ for all $i=1, \ldots, d$,
    \begin{align*}
        \overline{O_T(x)} = \overline{O_T(S_1^{h_1} \cdots S_d^{h_d} \psi_h (x))} = S_1^{h_1} \cdots S_d^{h_d}  \overline{O_T(\psi_h (x))}
    \end{align*}
    which, by the first paragraph, for $S_1^{h_1} \cdots S_d^{h_d} m_{X_\tau}$-almost every $x$, equals 
    \begin{equation} \label{eq s1 s2 supp mu_x}
        S_1^{h_1} \cdots S_d^{h_d} H \psi_h(x) = S_1^{h_1} \cdots S_d^{h_d}  \supp \mu_{\psi_h(x)} = \supp \mu_x.
    \end{equation}
    Thus, for $S_1^{h_1} \cdots S_d^{h_d} m_{X_\tau}$-almost every $x$, $\mu_x$ is ergodic. 
    
    Taking the disjoint pieces we have defined a map $x \mapsto \mu_x$ from $X$ to $\cM(X)$ that is continuous and such that by \eqref{eq disjoint part measure} fulfils that $\mu_x$ is ergodic for $\mu$ almost every $x \in X$. We conclude by noticing that for an arbitrary $f \in C(X)$
    \begin{align*}
        \int_X f \diff \mu &= \frac{1}{|D|} \sum_{h \in D} \int_{X_\tau}  S_1^{h_1} \cdots S_d^{h_d} f \diff m_{X_\tau} \\
       &= \frac{1}{|D|} \sum_{h \in D} \int_{X_\tau} \bigg( \int_X  f \diff [S_1^{h_1} \cdots S_d^{h_d}  \mu_y] \bigg) \diff m_{X_\tau}(y) 
        =\int_X \bigg( \int_X f \diff \mu_x \bigg) \diff \mu(x)
    \end{align*}
    concluding the proof. 
\end{proof}

\begin{proposition} \label{prop of the ergodic decomposition in nilsystems}
    Let $(X, \mu,T)$ be  a nilsystem. The continuous ergodic decomposition $x \mapsto \mu_x$ fulfills that
    \begin{itemize}
        \item  for all $x \in X$, $x \in \supp \mu_x$ and
        \item for all $x,y \in X$ $\supp \mu_x \cap \supp \mu_y \neq \emptyset$ if and only if $\supp \mu_x = \supp \mu_y$ if and only if $\mu_x = \mu_y$. 
    \end{itemize}
\end{proposition}

\begin{proof}
We still assume that  $X=G/\Gamma$ with $G = \langle G^\circ, \tau,\sigma_1,\ldots,\sigma_d \rangle$. By \cref{Continuous-ergodic-decomp-with-Leibman-group}, if $H$ denotes the Leibman group associated to $T$, we know that $H$ is $G_\tau=\langle G^\circ,\tau\rangle$-normal and that there is a finite set $D\subseteq \Z^d$ such that $\{\sigma_1^{h_1}\cdots \sigma_d^{h_d} X_\tau\}_{h\in D}$ partitions $X$ and if $x=S_1^{h_1}\cdots S_d^{h_d}x'$ with $x'\in X_\tau$ then $\mu_x= S_1^{h_1}\cdots S_d^{h_d}m_{Hx'}$ and therefore $\supp \mu_x = \sigma_1^{h_1}\cdots \sigma_{d}^{h_d}Hx'$. This immediately implies that  for all $x\in X$, $x\in \supp{\mu_x}$ and that $\supp{\mu_x}= \supp{\mu_y}$ if and only if $\mu_x=\mu_y$. The only thing left to prove thus is that if $\supp \mu_x \cap \supp \mu_y \neq \emptyset$, then $\supp \mu_x = \supp \mu_y .$  For $h,h'\in D$ write 
\begin{equation}\label{eq-support-measures}
  \supp{\mu_x}= \sigma_1^{h_1}\cdots \sigma_{d}^{h_d}Hx_0 ,\text{ and }\supp{\mu_y}= \sigma_1^{h_1'}\cdots \sigma_{d}^{h_d'}Hy_0   
\end{equation}
for where $x=S_1^{h_1}\cdots S_d^{h_d}x_0$ and $y=S_1^{h_1'}\cdots S_d^{h_d'}y_0$ with $x_0,y_0\in X_\tau$. Observe that $h=h'$ because $hX_\tau\cap h'X_\tau=\emptyset$ otherwise, which would contradict the hypothesis. Then, $\supp \mu_x \cap \supp \mu_y \neq \emptyset$ implies that 
$$Hx_0 \cap Hy_0\neq \emptyset. $$
Hence, we deduce that there are $h_x,h_y\in H$ such that $h_x x_0=h_yy_0$. In particular, we deduce that $y_0\in Hx_0$ and thus $Hy_0\subseteq Hx_0$. Similarly $Hx_0\subseteq Hy_0$ and then we have that $Hx_0=Hy_0$ and the conclusion of the proposition follows from \eqref{eq-support-measures}. 
\end{proof}

The following proposition is used later. 

\begin{proposition} \label{prop connected components kappa}
    Let $(X, \mu,T)$ be  a nilsystem. Let $x \mapsto \mu_x$ be the continuous ergodic decomposition of $\mu$ with respect to $T$. Then, there is $\kappa \in \N$ such that the number of connected components of $\supp \mu_x$ is $\kappa$ for all $x \in X$.
\end{proposition}

    \begin{proof}
        Taking, $H \leq G$ the Leibman group we get that $\supp \mu_{e_X} = H e_X$ and take $\kappa \in \N$ the number of connected components of $H e_X$.
        
        Using the notation from \cref{Continuous-ergodic-decomp-with-Leibman-group}, for $x \in X_\tau $ consider $i \in \N$ and $g \in G^\circ$ such that $x=\tau^i g e_X$, then by normality of $H$ on $\langle G^\circ , \tau \rangle $, 
        \begin{align*}
            \supp \mu_x = H x = \tau^i g H e_X = \tau^i g \supp \mu_{e_X}, 
        \end{align*}
        and since $\tau, g \in G$ induce homeomorphisms in $X$ we get that the number of connected components $ \supp \mu_x$ equals the one of $ \supp \mu_{e_X}$. For the general case, we conclude in a similar way using \eqref{eq s1 s2 supp mu_x}. 
    \end{proof}

\subsection{Continuity of ergodic decomposition in a subnilmanifold}
In this section, we show that continuity almost everywhere of the ergodic decomposition still holds for the measure of a subnilmanifold, under the right hypothesis. 

We  use the following result from Leibman.
\begin{theorem}[{ \cite[Theorem 2.2]{LEIBMAN_2007}}]\label{Leibman-generic-orbit}
Let $(X,\mu)$ be a nilmanifold with $X=G/\Gamma$
    Let $V$ be a closed connected subnilmanifold of $X$, let $K$ be a connected component of $\pi_\Gamma^{-1}(V)$ and $A$ be a closed subgroup of $G$. There exists a closed subnilmanifold $J_{V, A}$ of $X$ such that
    \begin{enumerate}
        \item for any $x \in V$ one has $\overline{\{ax : a\in A\}} \subseteq g J_{V, A}$ whenever $g \in K$ is such that $ \pi_\Gamma(g)=x$,  
        \item  there exists a zero $\mu_K$-measure set $P \subset K$-where $\mu_K$ denotes a Haar measure on $K$-such that for any $x \in V\setminus \pi_\Gamma(P)$ one has $\overline{\{ax : a\in A\}} =g J_{V, A}$ whenever $g \in K, \pi_\Gamma(g)=x$.
    \end{enumerate}
We call the subnilmanifold $J_{V, A}$ the generic orbit for $A$ on $V$; in the case $V=X$ the nilmanifold $J_{V, A}$ corresponds to the generic orbit for $A$ and is denoted by $J_A$.
\end{theorem}

Now we present the first case where we can ensure that the ergodic decomposition is continuous almost everywhere in a subnilmanifold. We start by assuming that the nilmanifold is connected, but we later relax those assumptions (see \cref{theorem cross ergodic decompositions}). 
\begin{proposition}\label{connected-centered-case}
  Let $(X=G/\Gamma,\mu,T)$ be a connected nilsystem and $Y=H e_X$ be a closed and connected subnilmanifold such that the orbit of $Y$ through $T$ is dense in $X$. Let $\nu$ be the Haar measure of $Y$ and let $x\mapsto \mu_x$ be the continuous ergodic decomposition of $\mu$ with respect to $T$. Then, $\nu$-a.e.  $x\in Y$, $\mu_x$ is ergodic.   
\end{proposition}
 \begin{proof}
    Since $X$ is connected, $G^\circ \Gamma = G$ because $G/\Gamma$ is connected.  For the continuous ergodic decomposition $x \mapsto \mu_x$, if $J' \subset X$ denotes the generic orbit from \cref{Leibman-generic-orbit} for $V= X$ (and connected component $G^\circ$), then $\mu_x = m_{gJ'}$ for $g \in G^\circ $ with $\pi_\Gamma(g)= x$ where $m_{gJ'}$ denotes the Haar measure of the nilmanifold $gJ'$. 

    Now, consider $Y = H e_X$ and let $K \subset G^\circ$ be a connected  component of $H$ containing $e_G$. Again $H = K \Gamma_H$ for $\Gamma_H = \Gamma \cap H$, because $Y$ is connected.  By \cref{Leibman-generic-orbit} applied to $V=Y$ (and connected component $K$), we have that there is a set $Y_0 \subset Y$ with $\nu(Y_0)=1$ such that for all $y\in Y_0$, $\overline{ O_T(y)}= gJ$ where $J$ is a subnilmanifold of $X$, and $\pi_\Gamma(g) =y$ with $g \in K$. By minimality of $\overline{ O_T(y)}$ we have that for all $y \in Y_0$, $g J \subset g J'$ whenever $\pi_\Gamma(g) = y$. With this $J \subset J'$,  To conclude the proposition we need to prove that $J = J'$, so that $g J = gJ'$ and with that for $y \in Y_0$ if $\pi_\Gamma(g) = y$, then measure $\mu_y = m_{gJ'} = m_{gJ}$ which is ergodic by construction. For that end we first prove the following claims,

    \underline{Claim 1:} $X= K J$.

    Notice that, for all $y \in Y$, $\overline{ O_T(y)} \subset gJ$ if $\pi_\Gamma(g) = y$ and $g \in K$ (this is again from \cref{Leibman-generic-orbit}).  Since the orbit of $Y$ along $T$ is dense in $X$ we get that 
    $$X=\overline{O_T(Y)}\subseteq \overline{KJ}. $$
    Thus, to conclude the claim it is enough to prove that $KJ$ is closed. For that, let $\Gamma_K=\Gamma\cap K$.  
    Since $K$ is rational, there is $C\subseteq K$ compact such that $K=C\Gamma_K$ (see \cite[Chapter 10, Lemma 14]{Host_Kra_nilpotent_structures_ergodic_theory:2018}). Moreover, pick $h \in C$ such that $h\Gamma$ has generic orbit, and pick $\gamma\in \Gamma_K$. Then  $hJ=\overline{O_T(h\Gamma)}=\overline{O_T(h\gamma \Gamma)} =h\gamma J$ and thus $J=\gamma J$. 
    We then get that $KJ = C \Gamma_K J = C J$ which is compact since $C$ and $J$ are. In particular, $KJ$ is closed and therefore $X=KJ$.

    For the second claim, let $L \subset G$ be a subgroup such that $J=L e_X$. Such $L$ exists because $J$ is a nilmanifold containing $e_X$.

    \underline{Claim 2:}  If we denote $L^\circ$ the connected component of $e_G$ in $L$, then $m_{G^\circ}(KL^\circ)>0$.

    Since $J=L e_X$ is compact and $L^\circ e_X$ open in $J$ (see \cref{preliminaries-2}), we get that $\{ l L^\circ e_X \colon l \in L\} $ is an open cover of $J$, then there is a finite $F \subset L$ such that  $\{lL^\circ e_X \}_{l\in F}$  cover $Le_X$. Moreover, since $L^\circ$ is normal in $L$ (see \cref{preliminaries-2}), we have that 
$$G=KL\Gamma = KFL^\circ \Gamma= KL^\circ F\Gamma. $$
Notice that $F\Gamma$ is countable, so if $KL^\circ$ has zero $m_{G}$-measure, then so does $KL^\circ F\Gamma=G$, which is not possible. Thus $m_{G}(KL^\circ)>0$,  and since $K L^\circ \subset G^\circ$, we deduce the claim.

Finally, using the characterization of $J'$ given by \cref{Leibman-generic-orbit} and since $m_{G^\circ}(KL^\circ)>0$, there are $k\in K$ and $l\in L^\circ$ such that $x=kl e_X$ has generic orbit $\overline{O_T(x)} = klJ'$. Notice that by definition $l e_X \in J$, hence $x = k l e_X \in k J$ and since $kJ$ is $T$-invariant we conclude that $klJ' = \overline{O_T(x)} \subset kJ$. Thus, $lJ' \subset J$ and since $J$ is shift invariant we deduce that $J' \subset J$, concluding the connected case.    
    \end{proof}

\begin{corollary} \label{cor con cambio de base}
  Let $(X=G/\Gamma,\mu,T)$ be a connected nilsystem and $Y=H y_0$ be a connected closed subnilmanifold such that the orbit of $Y$ through $T$ is dense in $X$. Let $\nu$ be the Haar measure of $Y$ and let $x\mapsto \mu_x$ be the continuous ergodic decomposition of $\mu$ with respect to $T$. Then, $\nu$-a.e.  $x\in Y$, $\mu_x$ is ergodic.   
\end{corollary}
\begin{proof}
The proof follows from doing a change of base point, see \cite[Chapter 10, section 2.4]{Host_Kra_nilpotent_structures_ergodic_theory:2018} for further details.

\end{proof}

We also show that we can remove the hypothesis of connectedness over $X$.
 \begin{corollary}\label{disconnected-connected-case}
  Let $(X=G/\Gamma,\mu,T)$ be a nilsystem and $Y=H y_0$ be a rational connected subnilmanifold such that the orbit of $Y$ through $T$ is dense in $X$. Let $\nu$ be the Haar measure of $Y$ and let $x\mapsto \mu_x$ be the continuous ergodic decomposition of $\mu$ with respect to $T$. Then, $\nu$-a.e.  $x\in Y$, $\mu_x$ is ergodic.       
 \end{corollary}

\begin{proof}
    First notice that $T \colon X \to X$ sends connected component of $X$ to connected components of $X$. By compactness $X$ has a finite number $r\in \N$ of connected components. Since $Y$ is connected then it is contained in one of this connected components $X_0$ whose Haar measure is denoted $\mu_0$. Moreover, since $\overline{O_T(Y)} = X$ then $\{T^i Y \colon 0 \leq i <r  \}$ meets every connected component once. In particular $X = \bigcup_{i=0}^{r-1} T^i X_0$. Similarly, $\overline{O_{T^r}(Y)} = X_0$ and since $X_0$ is connected we can use \cref{cor con cambio de base} to get that for $\nu$-almost every $y \in Y$, $\mu_{0,y}$ is $T^r$-ergodic, where $x \mapsto \mu_{0,x}$ is the continuous ergodic decomposition of $\mu_0$ with respect to $T^r$ in $X_0$. 

    Finally, since $\mu = \frac{1}{r} \sum_{i=0}^{r-1} T^{i}\mu_0$ the map $x =T^jx_0 \mapsto \mu_x =  \frac{1}{r} \sum_{i=0}^{r-1} T^{i}\mu_{0,x_0}$ where $x_0\in X_0$, is a continuous ergodic decomposition of $\mu$ that fulfills that for $\nu$-almost every $y \in Y$, $\mu_{y}$ is $T$-ergodic. 
\end{proof}

\begin{theorem} (Fiberwise continuous ergodic decomposition) \label{theorem cross ergodic decompositions}
    Let $(X=G/\Gamma, \mu,T,S)$ be an ergodic nilsystem. Let $x \mapsto \mu^T_x$ and $x \mapsto\mu^S_x$ be the continuous ergodic decomposition of $\mu$ for $T$ and $S$ respectively. For every $x \in X$ and $\mu^S_x$-almost every $y \in \supp \mu^S_x$ the measure $\mu^T_y$ is $T$-ergodic.  
\end{theorem}
\begin{proof}
    By standard reductions we assume that $G$ is generated by $G^\circ$, $\tau$ and $\sigma$, where $\tau$ and $\sigma$ are the elements inducing $T$ and $S$ respectively. 
    
    We want to reduce the proof to a simpler case where $x\in X_0=G^\circ e_X$, that is, the connected component for $e_X$ in $X$. To this end, consider $x \in X$ an arbitrary point, there are $g\in G^\circ$ and $i,j\in \Z$ such that $x=\tau^j\sigma^ig\Gamma$. By $S$-invariance, $\mu_{x}^S = \mu_{\tau^j g e_X}^S$. By \cref{Continuous-ergodic-decomp-with-Leibman-group}, we have that for any $z \in X$, $\mu_{\tau^j z}^S = T^j \mu_{z}^S$. This way, we get $\mu_{x}^S = T^j \mu_{ g e_X}^S$. Thus, using \cref{prop of the ergodic decomposition in nilsystems}, for $y\in \supp{\mu_{x}^S} $, we get that $T^{-j}y\in \supp{\mu_{g e_X}^S}$. This together with the fact that $\mu_{T^{-j}y}^T = \mu_{y}^T$ imply that showing $\mu_y^T$ is ergodic for 
    $\mu_x^S$-almost every $y$ is equivalent to showing that $\mu_{g e_X}^S$-almost every $y$, $\mu_y^T$ is ergodic. With this we have showed that the statement is reduced to prove the following claim

    \underline{Claim :} Under the assumptions of the theorem, for every $x \in X_0=G^\circ e_X$ and $\mu^S_x$-almost every $y \in \supp \mu^S_x$ the measure $\mu^T_y$ is $T$-ergodic.

    By assumption, $\supp{\mu_{e_X}^S}=H e_X$ with $H$ rational subgroup of $G$ that is normalized by $\langle G^\circ, \sigma\rangle$. By \cref{Continuous-ergodic-decomp-with-Leibman-group} we have that for every $x\in X_0$, $\supp{\mu_x^S}=Hx$.
    
    Fix $x\in X_0$ such that $\mu_x^S$ is ergodic. There exists $l\in \N$ such that $H^\circ x \sqcup \sigma H^\circ x  \sqcup \cdots \sqcup \sigma^{l-1} H^\circ x = Hx$. Let $z\in X_0$ be an arbitrary element. Then, by \cref{prop connected components kappa}, the support $\supp{\mu_z^S}$ has always $l$ connected components given by $H^\circ z ,\sigma H^\circ z  ,\cdots ,\sigma^{l-1} H^\circ z$. Set $Y'=H^\circ z$. Setting $X'$ as the ergodic nilsystem equal to $\overline{O_{\sigma^l,T}(e_X)}$, we have that $X'$ is clopen, as it is closed and finitely many shifts of $X'$ by $\sigma$ partition $X$ by minimality of $(T,S)$. We have that $Y'$ is a connected rational subnilmanifold of $X'$ which is $S^l$-invariant. In particular, we deduce then that the orbit of $Y'$ through $T$ is dense in $X'$. Then, we can apply \cref{disconnected-connected-case} to deduce that if $\nu_{Y'}$ is the Haar measure in $Y'$, then $\nu_{Y'}$-almost every $y\in Y'$, $m_y^{T}$ is ergodic where $x \mapsto m_x^T$ denotes the continuous $T$-ergodic decomposition for the Haar measure $m_{X'}$ of $X'$. Now, as $X'$ is clopen, for all $i =0,\ldots , l-1$, $S^i X'$ is a clopen set and therefore a finite union of connected components of $X$. Since $S$ maps connected components of $X$ to connected components of $X$, then the number of connected components of $S^i X'$ does not depend on $i =0,\ldots , l-1$ and moreover if $S^i X' \cap S^j X' \neq \emptyset$ then $S^i X' = S^j X'$. With this there is $m \in \N$ dividing $l$ such that 
    \begin{align*}
        X = \bigsqcup_{i=0}^{m-1} S^i X'
    \end{align*}
    which implies that $\mu = \frac{1}{m} \sum_{i=0}^{m-1} S^i m_{X'}$. Thus for $\nu_{Y'}$-almost every $y \in Y'$ we get that $\mu^T_y =  m_{y}^T$ is ergodic. We conclude the claim (and hence the proof) by noticing that $\mu_z^S=\frac{1}{l}\sum_{i=0}^{l-1} S^{i}\nu_{Y'}$.   
\end{proof}

\section{Hall-Petresco pronilsystem}
Using the terminology of \cite{kmrr25} (see also \cite{huang_shao_ye2026polynomial}), for a measure preserving system $\Xmt$ and $\ell \in \N$, we define the Furstenberg joining as the measure $\lambda \in \cM (X^{\ell+1})$ given by 
\begin{equation} \label{eq Furstenberg Joining}
    \int_{X^{\ell+1}} f_0 \otimes f_1 \otimes \cdots \otimes f_\ell \diff \lambda = \lim_{N \to \infty} \frac{1}{|\Phi_N|} \sum_{n\in \Phi_N} \int_X f_0 \cdot T^{n} f_1 \cdots T^{\ell n } f_\ell \diff \mu
\end{equation}
for all $f_0, f_1 , \ldots, f_\ell \in L^\infty(\mu)$ and a F\o lner sequence $\Phi$. Host and Kra proved in \cite{Host_Kra_nonconventional_averages_nilmanifolds:2005} that the choice of the F\o lner sequence does not change the measure $\lambda$. More importantly, they proved that one can replace the functions $f_i \in L^\infty(\mu)$ by $\E(f_i \mid Z_{\ell-1}) $ for $i =0, \ldots, \ell$, where $Z_{\ell-1}$ denotes the $(\ell-1)$-step pronilfactor of $(X,\mu,T)$. 

This measure is related to what we call the \emph{Hall-Petresco pronilsystem}. For $\ell \in \N$ fixed, if $(Z,m,T)$ is a minimal pronilsystem, we define the $\ell$-Hall-Petresco pronilsystem as
\begin{equation}\label{HP-nilsystem}
    \HP_\ell(Z) = \overline{\{ (Id \times T \times \cdots \times T^\ell)^n \Delta_Z \mid n \in \Z \} } \subset Z^{\ell+1}.
\end{equation}
If the choice of $\ell$ is explicit, we  refer to \eqref{HP-nilsystem} as simple the Hall-Petresco pronilsystem. 
This system has a unique $\langle \tilde T, T_\Delta \rangle$-invariant measure that we denote $m_{\HP_\ell (Z)}$, where we are denoting $\tilde T =Id \times T \times \cdots \times T^\ell$ and $T_\Delta = T \times \cdots \times T$.  

Notice that by construction and the mentioned result of Host and Kra (see also Ziegler \cite{ziegler2005nonconventional}) we get that for any ergodic system $\Xmt$, 
\begin{equation} \label{eq corollary of HK}
    \int_{X^{\ell+1}} f_0 \otimes  \cdots \otimes f_\ell \diff \lambda  = \int_{\HP_\ell (Z_{\ell-1})} \E(f_0 \mid Z_{\ell-1}) \otimes   \cdots \otimes \E(f_\ell \mid Z_{\ell-1}) \diff m_{\HP_\ell (Z_{\ell-1})}. 
\end{equation}
for all $f_0, \ldots, f_\ell \in L^\infty(\mu)$. This motivates the following lemma/definition.
\begin{lemma}
    Suppose that $\Xmt$ is an ergodic system with topological pronilfactors. Fix $\ell \in \N$ and consider $\lambda \in \cM(X^{\ell+1})$ as in \eqref{eq Furstenberg Joining} then $\lambda$ is supported in 
    \begin{equation*}
        \mathbb{A}_\ell (X) = (\pi_{\ell-1} \times \cdots \times \pi_{\ell-1})^{-1}( \HP_\ell(Z_{\ell-1})). 
    \end{equation*}
    In particular, $(\mathbb{A}_\ell (X), \lambda, Id\times T \times \cdots \times T^\ell, T \times \cdots \times T)$ is a measure preserving dynamical system. 
\end{lemma}

The set $\mathbb{A}_\ell (X)$ plays a similar role to $\mathbf{N}^{[\ell]}(X)$ in \cite{kmrr1}. 

\begin{proof}
    By continuity of $\pi_{\ell-1} \times \cdots \times \pi_{\ell-1}$, $\mathbb{A}_\ell (X)$ is closed. The fact that $\mathbb{A}_\ell (X)$ is $\langle \tilde T, T_\Delta\rangle$-invariant is direct from the invariance for $\HP_\ell(Z_{\ell-1})$ and the factor property coordinate-wise for $\pi_{\ell-1} \colon X \to Z_{\ell-1}$. Finally we have that
    \begin{equation*}
        \lambda( \mathbb{A}_\ell (X)) = \lambda((\pi_{\ell-1} \times \cdots \times \pi_{\ell-1})^{-1}\HP_\ell(Z_{\ell-1}))  = m_{\HP_\ell (Z_{\ell-1})} ( \HP_\ell (Z_{\ell-1}))=1. \qedhere
    \end{equation*} 
\end{proof}

\subsection{Continuous disintegration and ergodic decomposition in $\HP_\ell(X)$} \label{sec rho_z and sigma_z}

Recall from \cref{sec ergodic decomposition} the definition of continuous ergodic decomposition and continuous disintegration. Let $(X, \mu,T)$ be a system and $(Y,\nu,S)$ be a factor with continuous factor map $\pi \colon X \to Y$, if there is a continuous disintegration of $\mu$ over $\nu$, $y \mapsto \mu_y$, we say is {\em fully supported} if $x \in \supp \mu_y$ for all $x \in X$ with $\pi(x)=y$.  In particular, a continuous ergodic decomposition $x \mapsto \mu_x$ is fully supported if $x \in \supp \mu_x$ for all $x \in X$. This notion was extensively use in \cite{radic2026Uniformity}. 

\begin{remark} \label{rmk fully supp cont disintegration}
    In a fully supported continuous disintegration (and hence ergodic decomposition), $y \mapsto \supp \mu_y$ defines a partition of $X$, since $\supp \mu_y = \pi^{-1}(y)$ for all $y \in Y$.
\end{remark}

 \begin{lemma} \label{lemma ergodic fully supp decomposition}
     Every pronilsystem $(Z,m,T)$ admits a unique continuous and fully supported ergodic decomposition. 
 \end{lemma}

\begin{proof}
    This is direct from \cref{prop of the ergodic decomposition in nilsystems} and the fact that fully supported continuous ergodic decomposition is preserved under inverse limits \cite[Lemma 3.9]{radic2026Uniformity}.
\end{proof}

\begin{definition} \label{first def rho}
    Consider $(Z,m,T)$ a pronilsystem and $(\HP_\ell(Z), m_{\HP_\ell(Z)}, T_\Delta, \tilde T)$ the Hall Petresco pronilsystem. We define $w \mapsto \rho_{w} $ for $ w \in \HP_\ell(Z)$ the continuous ergodic decomposition of $m_{\HP_\ell(Z)}$ for the transformation $\tilde T$. We also denote, for $z \in Z$, $\rho_z = \rho_{(z,\ldots, z)}$.
\end{definition}

For a measure $\eta$ in a product space and a set of indices $L \subset \{0, \ldots,\ell\}$ we denote $\eta^L$ the projection of the measure to those coordinates and $L=\{i\}$ we simply denote it $\eta^i$. For the special case of $L = \{1, \ldots, \ell\}$, we denote $\eta^* = \eta^L$. Also for two measures $\eta, \gamma$ in a measurable space $\Omega$, we denote $\eta \leq C \gamma$ for some constant $C >0$ if $\eta(E) \leq C \gamma (E)$ for all measurables $E \subset \Omega$.  Recall that $\tilde T = Id \times T \times \cdots \times T^\ell$.

\begin{proposition} \label{prop marginals}
    Let $(Z,m,T)$ be an ergodic pronilsystem.   For all $w \in  \HP_\ell(Z)$, $\rho_{w}^0 = \delta_{w_0}$ and $\rho_{w}^i \leq i m$ for all $i = 1, \ldots, \ell$. 
    \end{proposition}

\begin{proof}
    Notice that  $m$-almost every $w$, $\rho_w$ is ergodic, hence by properties of pronilsystems $(\supp \rho_w, \tilde T)$ is a minimal system (see \cref{preliminaries-2}). In particular, if $\tilde w \in \supp \rho_w$, then $\tilde w_0 = w_0$. This already ensures that $\rho_w^0 = \delta_{w_0}$ for $\lambda$-almost all $w \in \HP_\ell(Z) $ but, by continuity of the ergodic decomposition, we conclude it for all $ w \in \HP_\ell(Z)$. Similarly, since for $\lambda$-almost every $w \in \HP_\ell(Z)$, $(\supp \rho_w,\tilde{T})$ is uniquely ergodic, we get that for those $w \in \HP_\ell(Z)$ 
    \begin{equation} \label{eq casi def of sigma}
        \rho_w = \lim_{N \to \infty } \frac{1}{N} \sum_{n =1}^N \tilde T^n \delta_w  \implies \rho_w^i = \lim_{N \to \infty } \frac{1}{N} \sum_{n =1}^N T^{in} \delta_{w_i}
    \end{equation}
    for all $i =1, \ldots, \ell$, in particular $\rho_{w}^i \leq i m$. Again, by the continuity of the ergodic decomposition we derive that $\rho_{w}^i \leq i m$ for all $w \in \HP_\ell(Z)$. 
\end{proof}

From \eqref{eq casi def of sigma}, one is tempted to define $\rho_w$ directly as that limit. However, that equality only holds $\lambda$-almost surely and it is false in general. 
Examples of this can be found in \cite[Appendix A]{hernandez_kousek_radic2025density_Hindman} and \cite[Section 8]{radic2026Uniformity}. In any case, measures defined as in \eqref{eq casi def of sigma} play a special role in the theory and in this work. 

\begin{definition} \label{defi xi}
    Let $(Z, m,T)$ be a pronilsystem and let $z \in Z$, then we define 
    \begin{equation*}
        \sigma_z=\lim_{N \to \infty}\frac{1}{N} \sum_{n = 1}^N \tilde T^n (\delta_{z} \times \cdots \times \delta_z).
    \end{equation*}
    Notice that $\supp \sigma_z \subset \supp \rho_z$ for all $z \in Z$ (see \cref{lemma ergodic fully supp decomposition}). 
\end{definition}

The following theorem was originally proved for the nilsystem case (see \cite[Theorem 7 Chapter 15]{Host_Kra_nilpotent_structures_ergodic_theory:2018}), here we adapt the argument for general pronilsystem and we simplify the proof using the continuity of $z \mapsto \rho_z$. 


\begin{theorem} \label{thrm new cont ergodic decomp} Let $(Z,m,T)$ be an ergodic pronilsystem and $\ell \in \N$,
    \begin{enumerate}[i)]
        \item \begin{equation} \label{eq second ergodic decomposition}
            m_{\HP_\ell(Z)} = \int_Z \rho_z \diff m(z),
        \end{equation}
        \item \label{thrm point 2} for every $z \in Z$ and $w \in \HP_\ell(Z)$, $\rho_z = \rho_w$ if and only if $w_0 = z$, and
        \item \label{thrm point 3} for $m$-almost every $z \in Z$, $\rho_z = \sigma_z$. In particular,  $\supp \rho_z$ is uniquely ergodic with $\tilde T$-invariant measure $\rho_z$ for $m$-almost every $z \in Z$.  
    \end{enumerate}
\end{theorem}

\begin{proof}
    First, set $\nu = \int_Z \rho_z \diff m(z)$. We have to prove that $\nu= m_{\HP_\ell(Z)} $. Clearly, $\supp \nu \subset \HP_\ell(Z)$ and since $\rho_z$ is $\tilde T$-invariant for all $z \in Z$, so is $\nu$. Using the $T$-invariance of $m$
    \begin{equation*}
        \nu = \int_Z \rho_z \diff m (z) = \int_Z \rho_{Tz} \diff m (z) = \int_Z T_\Delta \rho_{z} \diff m (z) = T_\Delta \nu.
    \end{equation*}
    Thus $\nu$ is $\langle T_\Delta, \tilde T\rangle$-invariant and supported in $\HP_\ell(Z)$ which is $\langle T_\Delta, \tilde T\rangle$-uniquely ergodic, so we conclude the equality. 

 For \eqref{thrm point 2}, if $\rho_z=\rho_w$ for $z\in Z$, $w\in HP_\ell(Z)$, then by \cref{prop marginals} we get $w_0=z$. For the other implication, notice that, by \cref{rmk fully supp cont disintegration} and \cref{lemma ergodic fully supp decomposition}, for $w, w' \in \HP_\ell (Z)$, $w \in \supp \rho_{w '}$ if and only if $\rho_w = \rho_{w'}$. Thus, since by \cref{prop marginals}, for all $w \in \HP_\ell(Z)$, $$\supp \rho_w \subset \{ u \in \HP_\ell(Z) \colon u_0 = w_0\},$$ we get that $\rho_z$-almost every $w$ we have that $\rho_w=\rho_{P_0(w)}$, where $P_0 \colon HP_\ell(Z) \to Z$ is the projection to the $0$ coordinate. 
 Integrating and using \eqref{eq second ergodic decomposition}, we conclude that $m_{\HP_\ell(Z)}$-almost every $w$, we have $\rho_w=\rho_{P_0(w)}$ and by continuity this equality holds for every $w\in \HP_\ell(Z)$.


    For \eqref{thrm point 3}, first notice that with the previous point, $z \mapsto \rho_z$ defines a continuous ergodic decomposition. Let $W_z = \overline{\cO_{\tilde T}(z,\ldots, z)} $ with its unique $\tilde T$-invariant measure $\sigma_z$. 
    By unique ergodicity in $(W_z,\tilde{T})$ and the dominated convergence theorem, we have that, for any $f_0, \ldots, f_\ell \in C(Z)$
    \begin{align*}
       \int_Z &\int_{\HP_\ell(Z)}  \bigotimes_{j=0}^\ell f_j \diff \sigma_z \diff m(z) = \int_Z \lim_{N \to \infty} \frac{1}{N} \sum_{n=0}^{N-1} \prod_{j=0}^\ell f_j(T^{jn}z)  \diff m(z) \\
       &= \lim_{N \to \infty} \frac{1}{N} \sum_{n=0}^{N-1} \int_Z  \prod_{j=0}^\ell f_j(T^{jn}z) \diff m(z) = \int_{\HP_\ell(Z)}  \bigotimes_{j=0}^\ell f_j \diff m_{\HP_\ell(Z)}.
    \end{align*}
    Therefore, $z \mapsto \sigma_z$ defines a (non-continuous) ergodic decomposition of $m_{\HP_\ell(Z)}$ with this $\rho_z = \sigma_z$ for $m$-almost all $z \in Z$ concluding the proof. 
\end{proof}

\subsection{Good projection for pronilsystems} \label{sec good projections}

In this section we  prove that in a pronilsystem $(Z, m,T)$, for a fixed $\ell \in \N$ and $z \in Z$, if $\rho_z \in \Lambda_\ell(Z)$ is the measure defined in \cref{first def rho}, then for any $L = \{\ell_1, \ldots, \ell_k\} \subset \{1,\ldots, \ell\}$ the measure $\rho^L_z$ can be defined as a lift of a measure in $Z_{k-1}^k$ where $Z_{k-1}$ is the $(k-1)$-step pronilfactor of $Z$. For that, we use the following lemmas originally proved in \cite{radic2026Uniformity}.

\begin{lemma}[{\cite[Propositions 3.2 and 3.6]{radic2026Uniformity}}]\label{prop cont disintegration and conditional expectation}
    If $\Xmt$ admits a continuous disintegration $y \mapsto \mu_y$ over a factor system $(Y,\nu,T)$, then for every $f \in C(X)$ the conditional expectation $\E(f \mid Y)$ agrees $\nu$-almost surely with a continuous function.
\end{lemma}

\begin{lemma}[{\cite[Corollary 3.12]{radic2026Uniformity}}]\label{prop cont k-step max pronilfactors}
    Let $(Z, m,T)$ be an ergodic $s$-step pronilsystem, then for every $k \leq s$, $(Z, m,T)$ admits a fully supported continuous disintegration over $Z_{k}$.
\end{lemma}

\begin{proposition} (Good projection property)  \label{prop good projection properties}
    Let $(Z, m ,T)$ be an ergodic pronilsystem, $2 \leq k \leq \ell$ fixed natural numbers and $Z_{k-1}$ its $(k-1)$-step pronilfactor. Denote $\rho_z \in \HP_\ell(Z)$ and $\eta_u \in \HP_\ell (Z_{k-1})$ the continuous ergodic decomposition of, respectively, $m_{HP_\ell(Z)}$ and $m_{HP_\ell(Z_{k-1})}$ for the respective transformation $\tilde{T}$ (see \cref{first def rho}). For all $L = \{ \ell_1, \ldots, \ell_k \} \subset \{1, \ldots, \ell\}$, $z \in Z$ and $f_1, \ldots, f_k \in C(Z)$ 
    \begin{equation}  \label{eq good projection}
        \int_{Z^k} f_1 \otimes \cdots \otimes f_k \diff \rho_z^L = \int_{Z^k_{k-1}} \E(f_1\mid Z_{k-1}) \otimes \cdots \otimes \E(f_k \mid Z_{k-1}) \diff \eta^L_{\pi_{k-1}(z)}
    \end{equation}
\end{proposition}

\begin{proof}

    As previously, for $z \in Z$ and $u \in Z_{k-1}$, let $ \sigma_z $ and $\nu_u$ be the unique invariant measures of $\overline{\cO_{\tilde T}(z, \ldots,z)}$ and $\overline{\cO_{\tilde T}(u, \ldots,u)}$ respectively. Then by unique ergodicity and by the main result of Host and Kra in \cite{Host_Kra_nilpotent_structures_ergodic_theory:2018}, for all $f_1, \ldots, f_k \in C(Z)$, we get that
    \begin{align*}
        \norm{ \lim_{N \to \infty}  \frac{1}{N} \sum_{n=0}^{N-1} \prod_{i=1}^k T^{\ell_in} f_i(\bullet)   - \int_{Z^L} \bigotimes_{i=1}^k f_i  \diff \sigma_{\bullet}^L}_{L^2(m)} =0 \\
        \norm{ \lim_{N \to \infty}\frac{1}{N}  \sum_{n=0}^{N-1}  \prod_{i=1}^k T^{\ell_in} f_i(\bullet)  - \lim_{N \to \infty} \frac{1}{N}  \sum_{n=0}^{N-1}  \prod_{i=1}^k T^{\ell_in} \E(f_i \mid Z_{k-1})(\pi_{k-1}(\bullet)) }_{L^2(m)} =0 \\
        \norm{ \lim_{N \to \infty} \frac{1}{N}  \sum_{n=0}^{N-1}  \prod_{i=1}^k T^{\ell_in} \E(f_i \mid Z_{k-1})(\pi_{k-1}(\bullet)) - \int_{Z^L_{k-1}} \bigotimes_{i=1}^k \E(f_i \mid Z_{k-1})  \diff \nu_{\pi_{k-1}(\bullet)}^L}_{L^2(m)} =0
    \end{align*}
    where we represent the integration variable with $\bullet$ for clarity, and in the last equality we are using that $u \mapsto \E(f_i \mid Z_{k-1})(u)$ agrees with a continuous function in a set of full measure by Lemmas \ref{prop cont disintegration and conditional expectation} and \ref{prop cont k-step max pronilfactors}. Then for a countable dense family $\cF$ of functions in $C(Z)$ we get that there is a set $Z' \subset Z$ with $m(Z')=1$ such that for all $z \in Z'$ and all $f_1, \ldots , f_k \in \cF$
    \begin{equation} \label{eq equality aux measure}
        \int_{Z^L} \bigotimes_{i=1}^k f_i  \diff \sigma_z^L = \int_{Z^L_{k-1}} \bigotimes_{i=1}^k \E(f_i \mid Z_{k-1})  \diff \nu_{\pi_{k-1}(z)}^L.  
    \end{equation}
    For $u \in Z_{k-1}$ we define $\breve \nu_u, \breve \eta_u \in \cM(Z^k)$ the measures defined by 
    \begin{align*}
        \int_{Z^L} \bigotimes_{i=1}^k f_i  \diff \breve \nu_u = \int_{Z^L_{k-1}} \bigotimes_{i=1}^k \E(f_i \mid Z_{k-1})  \diff \nu_{u}^L \\
        \int_{Z^L} \bigotimes_{i=1}^k f_i  \diff \breve \eta_u = \int_{Z^L_{k-1}} \bigotimes_{i=1}^k \E(f_i \mid Z_{k-1})  \diff \eta_{u}^L
    \end{align*}
    for all $f_1, \ldots, f_k \in C(Z)$. Then by \eqref{eq equality aux measure}, $\sigma_z^L = \breve \nu_{\pi_{k-1}(z)}$ for $m$-almost every $z \in \Z$. Moreover, by  \cref{thrm new cont ergodic decomp}, for $z$ almost every $z \in Z$, $\rho_z^L = \sigma_z^L$ and $\breve \nu_{\pi_{k-1}(z)} = \breve \eta_{\pi_{k-1}(z)}$, hence $\rho_z^L =\breve \eta_{\pi_{k-1}(z)}$ for $m$-almost every $z \in Z$. 
    
    Using again that $u \mapsto \E(f_i \mid Z_{k-1})(u)$ and that $u \mapsto \eta_u$ are continuous map, we get that $u \mapsto \breve \eta_u$ is continuous and since $\rho_z^L =\breve \eta_{\pi_{k-1}(z)}$ in a dense set of $Z$, by continuity we conclude that $\rho_z^L =\breve \eta_{\pi_{k-1}(z)}$ for all $z \in Z$.   
\end{proof}

\subsection{Progressive measures}\label{progressive measures}

Following \cite{kmrr25}, we say that a measure $\tau \in \cM(X^{\ell+1})$ is progressive if for all open sets $E_1, \ldots, E_\ell \subset X$ with
\begin{equation*}
    \tau(X \times E_1 \times \cdots \times E_\ell) >0,
\end{equation*}
then there are infinitely many $n \in \N$ such that
\begin{equation*}
    \tau((X \times E_1 \times \cdots \times E_\ell ) \cap T_{\Delta}^{-n}(E_1 \times \cdots \times E_\ell \times X)) >0. 
\end{equation*}

Extending this notion we define:
\begin{definition}
     We say that a  measure $\tau \in \cM(X^{\ell+1})$ is uniformly progressive along $\Phi$ if for every $\eta >0$ there exists $\delta >0$ such that for every continuous and non-negative functions $f_1, \ldots, f_\ell \in C(X)$ bounded by $1$ whenever
    \begin{equation*}
    \int_{X^{\ell+1}}\1 \otimes f_1 \otimes \cdots \otimes f_\ell \diff \tau>\eta,
\end{equation*}
then 
\begin{equation}\label{conclusion-Uniform-progressiveness}
     \limsup_{N \to \infty }\frac{1}{|\Phi_N|} \sum_{n \in \Phi_N} \int_{X^{\ell+1}}(\1 \otimes f_1 \otimes \cdots \otimes f_\ell) \cdot T_{\Delta}^n (f_1 \otimes \cdots \otimes f_\ell \otimes \1) \diff \tau \geq \delta.
\end{equation}   
\end{definition}

\begin{theorem} \label{theorem key}
    Fix $s,\ell \in \N$. Let $Z$ be an ergodic pronilsystem. For $(\HP_\ell(Z), m_{\HP_\ell(Z)}, \tilde T, T_\Delta)$ consider $z \mapsto \rho_z$ and $w \mapsto \lambda_w$ the continuous $\tilde T$-ergodic and $T_\Delta$-ergodic decompositions of $m_{\HP_\ell(Z)}$ respectively. For any $z \in Z$ and $\rho_z$-almost every $w \in \HP_\ell(Z)$, the measure $\lambda_w$ is ergodic.  
\end{theorem}

\begin{proof}
    The case $Z = G/\Gamma$ being and $s$-step nilsystem is direct from \cref{theorem cross ergodic decompositions} and the fact that  $\HP_\ell(G)=\langle\HP_\ell(G)^\circ, \tilde \tau, \tau_\Delta \rangle$, (see \cite[proof of Theorem 5, Chapter 15]{Host_Kra_nilpotent_structures_ergodic_theory:2018}). 
    
    For the inverse limit case, let $(Z_i, m_i, T)$ be the sequence of $s$-step nilsystem such that $(Z, m, T) = \lim_{\leftarrow} (Z_i, m_i, T)$ and denote $\theta_i \colon Z \to Z_i$ the factor map. We also denote $\tilde \theta_i = \theta_i \times \cdots \times \theta_i$ that maps $\HP_\ell(Z)$ to $\HP_\ell(Z_i)$.
    Fix $z \in Z$. Since the span of $$\{ f \circ \tilde \theta_i \colon f \in C(\HP_\ell(Z_i)), i \in \N\}$$ is dense in $C(\HP_\ell(Z))$ and $C(\HP_\ell(Z_i))$ is separable, it is enough to prove that for $\rho_z$-almost every $w \in \HP_\ell(Z)$, if $f \in C(\HP_\ell(Z_i))$ then $\lim_{N \to \infty} \frac{1}{N} \sum_{n =1}^N f\circ \tilde \theta_i(T^n_\Delta w) = \int f \circ \tilde \theta_i \diff \lambda_w$. 
    
    For $i \in \N$, denote $u \in Z_i \mapsto \rho^{Z_i}_u$ the $\tilde T$-continuous ergodic decomposition of $m_{\HP_\ell(Z_i)}$. Similarly for $\lambda_w^{Z_i}$.  Let $W_i \subset \HP_\ell(Z_i)$ be the measurable set with $\rho_{\theta_i(z)}^{Z_i}(W_i) =1$ such that $\lambda_w^{Z_i}$ is ergodic for all $w \in W_i$. Notice that $\rho_{\theta_i(z)}^{Z_i}(W_i) = \rho_{z}(\tilde{\theta}_i^{-1}(W_i))$ and therefore the set $$ W= \bigcap_{i \in \N} \tilde{\theta}_i^{-1}(W_i)$$ has full measure for $\rho_z$. Then for $w \in W$ and $f \in C(\HP_\ell(Z_i))$ for some $i\in \N$
    \begin{equation*}
        \lim_{N \to \infty} \frac{1}{N} \sum_{n \leq N} f\circ \tilde \theta_i(T^n_\Delta w) = \lim_{N \to \infty} \frac{1}{N} \sum_{n \leq N} f(T^n_\Delta \tilde\theta_i( w)) = \int_{\HP_\ell(Z_i)} f \diff \lambda_{\tilde\theta_i(w)}^{Z_i} = \int_{\HP_\ell(Z)} f \circ \tilde \theta_i \diff \lambda_{w}.
    \end{equation*}
    In particular, we have that $\lambda_{\tilde\theta_i(w)}^{Z_i}= (\tilde\theta_i)_*\lambda_{w}$ from which the ergodicity of $\lambda_w$ follows, concluding the proof. 
\end{proof}

To conclude, we cite the following result from \cite{hernandez2025infinite}.

\begin{proposition}[{\cite[Proposition 5.5]{hernandez2025infinite}}]\label{uniform-Szemeredi-0}  Let $\ell,s\in \N$. Let $Z$ be an ergodic $s$-step pronilsystem. The measure $\sigma_z \in \cM(Z^{\ell+1})$ is uniformly progressive along any F\o lner sequence for all $z \in Z$. Moreover, the constant associated to progressiveness in \eqref{conclusion-Uniform-progressiveness} depends only in $\ell$ and $\eta>0$. 
\end{proposition}
\begin{remark} \cite[Proposition 5.5]{hernandez2025infinite} is proven for nilsystems, but an approximation argument gives the statement for pronilsystems. Similarly, in \cite[Proposition 5.5]{hernandez2025infinite} it is not explicit that the  constant depends only in $\eta$ and $\ell$, but it follows immediately from the proof, as it relies solely in the use of uniform Szemerédi's theorem \cite[Theorem A.2]{bergelson_Kulaga_Lemanczyk_Richter2019rationally} and not in the properties of the subjacent nilsystem.

\end{remark}

\begin{theorem} \label{thrm rho uniformly progresive} \label{thrm rho is unif progressive}
   Let $\ell,s\in \N$. Let $(Z,m,T)$ be an ergodic $s$-step pronilsystem. The measure $\rho_z \in \cM(\HP_\ell(Z))$ is uniformly progressive for all $z \in Z$.  
\end{theorem}

\begin{proof}
    Fix $f_1, \ldots, f_\ell \in C(Z)$ continuous non-negative functions bounded by $1$ and $z \in Z$ with $$\int_{\HP_\ell (Z)} \1 \otimes f_1 \otimes \cdots \otimes f_\ell \diff \rho_z  > \eta.$$ Using \cref{thrm new cont ergodic decomp}, consider $D \subset Z$ the dense set with $m(D) =1$ such that $\rho_u = \sigma_u$ for all $u \in D$. If $z \in D$ then we conclude by \cref{uniform-Szemeredi-0}. 
    
    Suppose then that $z \not \in D$. By continuity of $z \mapsto \rho_z$, there is an open neighborhood $V$ of $z$ such that $$\int_{\HP_\ell (Z)} \1 \otimes f_1 \otimes \cdots \otimes f_\ell \diff \rho_u  > \eta$$ for all $u \in V$. Using that $D \cap V$ is dense in $V$ (and thus non-empty), by \cref{uniform-Szemeredi-0} for all $u \in D \cap V$  
    \begin{equation} \label{eq limit progresive}
        \liminf_{N \to \infty} \frac{1}{|\Phi_N|} \sum_{n \in \Phi_N} \int_{\HP_\ell(Z)} (\1 \otimes f_1 \otimes \cdots \otimes f_\ell ) T_{\Delta}^n (f_1 \otimes \cdots \otimes f_\ell \otimes \1) \diff \rho_u \geq \delta.
    \end{equation}

    \underline{Claim} The map from $Z$ to $\R$ given by
    \begin{equation} \label{eq limit progresive 2} 
        u \mapsto \liminf_{N \to \infty} \frac{1}{|\Phi_N|} \sum_{n \in \Phi_N} \int_{\HP_\ell(Z)} (\1 \otimes f_1 \otimes \cdots \otimes f_\ell ) T_{\Delta}^n (f_1 \otimes \cdots \otimes f_\ell \otimes \1) \diff \rho_u 
    \end{equation}
    is continuous. 

    To prove the claim fix $u \in Z$ arbitrary. By \cref{theorem key}, for $\rho_u$ almost every $w \in \HP_\ell(Z)$, $\lambda_w$ is ergodic and therefore uniquely ergodic in its support, so by dominated convergence theorem the limit in \eqref{eq limit progresive 2} equals
    \begin{align*}
          &\int_{\HP_\ell(Z)} (\1 \otimes f_1 \otimes \cdots \otimes f_\ell ) \lim_{N \to \infty} \frac{1}{|\Phi_N|} \sum_{n \in \Phi_N} T_{\Delta}^n (f_1 \otimes \cdots \otimes f_\ell \otimes \1 ) \diff \rho_u \\ 
         = & \int_{\HP_\ell(Z)} \Big(\1 \otimes f_1 \otimes \cdots \otimes f_\ell \Big)(w) \Big( \int_{\HP_\ell(Z)}  f_1 \otimes \cdots \otimes f_{\ell-1} \otimes f_\ell \otimes \1 \diff \lambda_w \Big) \diff \rho_u(w)
    \end{align*}
    Since $f_1 \otimes \cdots \otimes f_\ell \otimes \1$ is a continuous function the map $w \mapsto \int  f_1 \otimes \cdots  \otimes f_\ell \otimes \1 \diff \lambda_w $ is also continuous (by continuity of $w \mapsto \lambda_w$). By the same reason the map 
    \begin{equation*}
        u \mapsto \int_{\HP_\ell(Z)} \Big(\1 \otimes f_1 \otimes \cdots \otimes f_\ell \Big)(w) \Big( \int_{\HP_\ell(Z)}  f_1 \otimes \cdots \otimes f_\ell \otimes \1 \diff \lambda_w \Big) \diff \rho_u(w) 
    \end{equation*}
    is continuous, concluding the claim. 

    Finally, using the claim, since the function in \eqref{eq limit progresive} is continuous and bounded by $\delta$ for all $u \in D\cap V$ which is dense in $V$, we conclude the property for all $u \in V$ in particular for $z$.
\end{proof}

Another consequence from \cref{theorem cross ergodic decompositions}, used in the proof of \cref{main thrm nilborh}, is the following,

\begin{corollary} \label{cor to get shift t}
Let $\ell,s\in \N$. Let $(Z,m,T)$ be an ergodic $s$-step pronilsystem. For every $z \in Z$, F\o lner sequence $\Phi$ and $f_1, \ldots , f_\ell \in C(Z)$
    \begin{equation} \label{eq for finding t}
          \lim_{N \to \infty} \int_{\HP_\ell(Z)} \frac{1}{|\Phi_N|} \sum_{n \in \Phi_N}   T_{\Delta}^n (\1 \otimes f_1 \otimes \cdots \otimes f_\ell ) \diff \rho_z = \int_{\HP_\ell(Z)} (\1 \otimes f_1 \otimes \cdots \otimes f_\ell)  \diff m_{\HP_\ell(Z)}
    \end{equation}
\end{corollary}

\begin{proof}
    By \cite[Theorem 5.2]{kmrr25}, \eqref{eq for finding t} is already known for $z \in Z$ such that $\rho_z$ is ergodic. Notice that the right hand side does not depend on $z$, and similarly to the proof \eqref{eq limit progresive 2}, by \cref{theorem cross ergodic decompositions}, the function from $Z$ to $\R$ given by
    \begin{equation} \label{eq cont}
        u \mapsto \lim_{N \to \infty} \int_{\HP_\ell(Z)} \frac{1}{|\Phi_N|} \sum_{n \in \Phi_N}   T_{\Delta}^n (\1 \otimes f_1 \otimes \cdots \otimes f_\ell ) \diff \rho_u
    \end{equation}
    is continuous. Using that the points $z \in Z$ such that $\rho_z$ is ergodic are dense in $Z$ and the continuity of the map in \eqref{eq cont}, we conclude the equality \eqref{eq for finding t} for all $z \in Z$. 
\end{proof}

\subsection{Lifting progressive measures}

With the two families of measures in $\HP_\ell(Z)$, $z \mapsto \rho_z$ and $z\mapsto \sigma_z$, we define two families of measures in $\A_\ell(X)$ by lifting them.
\begin{definition} \label{def rho and sigma tilde}
    Let $\Xmt$ be an ergodic system with topological pronilfactors and $a \in X$. Considering $Z = Z_{\ell-1}(\mu)$ in Definitions \ref{first def rho} and \ref{defi xi}, for all $f_0, f_1, \ldots, f_\ell \in C(X)$ we define the measures $\tilde \rho_a$ and $\tilde \sigma_a$ given by
    \begin{align*}
        \int_{\A_{\ell}(X)} f_0 \otimes f_1 \otimes \cdots \otimes f_\ell \diff \tilde \rho_a = f_0(a) \int_{Z_{\ell-1}^\ell} \E(f_1| Z_{\ell-1 }) \otimes \cdots \otimes \E(f_\ell| Z_{\ell-1 }) \diff \rho_{\pi_{\ell-1}(a)}^* \\
        \int_{\A_{\ell}(X)} f_0 \otimes f_1 \otimes \cdots \otimes f_\ell \diff\tilde \sigma_a = f_0(a) \int_{Z_{\ell-1}^\ell} \E(f_1| Z_{\ell-1 }) \otimes \cdots \otimes \E(f_\ell| Z_{\ell-1 }) \diff \sigma_{\pi_{\ell-1}(a)}^*
    \end{align*}
\end{definition}

\begin{theorem}\label{progressive}
    Let $\Xmt$ be an ergodic system with topological pronilfactors and let $\ell \in \N$. Given a F\o lner sequence $\Phi$ and $a \in \gen(\mu,\Phi)$, the measure $\tilde{\rho}_a \in  \cM(\A_\ell (X))$ for $a \in \gen(\mu, \Phi)$ is uniformly progressive along $\Phi$. Moreover, the constant $\delta >0$ in the definition only depends on $\eta >0$ and $\ell$, but does not depend on the system $\Xmt$.
\end{theorem}

\begin{proof}
    Consider $f_1, \ldots, f_\ell \in C(X)$ non-negative continuous functions bounded by $1$. Assume that we have that
    \begin{equation*}
        \int_{\A_\ell (X)} \1 \otimes f_1 \otimes \cdots \otimes f_\ell \diff \tilde \rho_a > 2\eta. 
    \end{equation*}
    By definition of the measure $\tilde \rho_a$, for $Z=Z_{\ell-1}$ the $(\ell-1)$-step pronilfactor of $X$,
    \begin{equation*}
        \int_{\HP_\ell (Z)} \1 \otimes \E(f_1|Z) \otimes \cdots \otimes \E(f_\ell|Z) \diff \rho_{\pi(a)} =\int_{\A_\ell (X)} \1 \otimes f_1 \otimes \cdots \otimes f_\ell \diff \tilde \rho_a > 2\eta. 
    \end{equation*}
    where $\pi \colon X \to Z$ is the continuous factor map. Consider $0< \delta < \eta$ the constant of uniform progressiveness for the family $\rho_z \in \cM(Z^{\ell+1})$ depending on $\eta>0$ and $\ell$. By density, we can consider $g_1, \ldots, g_\ell \in C(Z)$ non-negative continuous functions bounded by $1$ such that $\norm{\E(f_i | Z) - g_i}_{L^2 (m)} \leq \delta/(4\ell^2)$,  for all $i = 1, \ldots, \ell$. We can also suppose that $\norm{\E(f_1 | Z) - g_1}_{L^1 (m)} \leq \delta/4$. With this,we get that
    \begin{equation*}
         \int_{\HP_\ell (Z)} \1 \otimes g_1 \otimes \cdots \otimes g_\ell \diff \rho_{\pi(a)} > \eta.
    \end{equation*}
   and hence,
    \begin{align*}
        &\liminf_{N \to \infty} \frac{1}{|\Phi_N|}\sum_{n \in \Phi_N}\int_{\A_\ell(X)} (1\otimes f_1 \otimes \cdots \otimes f_\ell) \cdot T_\Delta^n (f_1 \otimes \cdots \otimes f_\ell \otimes \1 ) \diff \tilde \rho_a \\
        &=\liminf_{N \to \infty} \frac{1}{|\Phi_N|} \sum_{n \in \Phi_N} f_1(T^na) \int_{\HP_\ell(Z)} \prod_{i=1}^\ell  \Big( \E(f_i \mid Z) \cdot T^n \E(f_{i+1} \mid Z) \Big)^{[i]}\diff \rho_{\pi(a)}  \\
        &\geq -\delta/2 + \liminf_{N \to \infty} \frac{1}{|\Phi_N|} \sum_{n \in \Phi_N} f_1(T^na) \int_{\HP_\ell(Z)} \prod_{i=1}^\ell  \Big( g_i \cdot T^n g_{i+1} \Big)^{[i]}\diff \rho_{\pi(a)}, 
    \end{align*}
    where we imposed $f_{\ell+1}=g_{\ell+1} = \1$, and we applied \cite[Theorem 6.6]{kmrr25} and the definition of $\tilde{\rho}_a$ in the first equality.
    To use uniform progressivity of $\rho_{\pi(a)}$ we notice that, by Fatou's lemma
    \begin{align*}
        &\bigg|\limsup_{N \to \infty} \frac{1}{|\Phi_N|} \sum_{n \in \Phi_N} (f_1(T^na)- g_1(T^n \pi(a))) \int_{\HP_\ell(Z)} \prod_{i=1}^\ell  \Big( g_i \cdot T^n g_{i+1} \Big)^{[i]}\diff \rho_{\pi(a)}  \bigg| \\
         &\leq  \int_{\HP_\ell(Z)} \bigg| \limsup_{N \to \infty} \frac{1}{|\Phi_N|} \sum_{n \in \Phi_N} (f_1(T^na)- g_1(T^n \pi(a))) \prod_{i=1}^\ell  \Big( g_i \cdot T^n g_{i+1} \Big)^{[i]} \bigg|\diff \rho_{\pi(a)} \\
        & \leq \norm{\E(f_1 \mid Z) - g_1}_{L^1 (m)} \leq \delta/4
    \end{align*} 
    where in the last step we are using \cite[Lemma 6.11]{kmrr25}. With this we get that
    \begin{align*}
        &\liminf_{N \to \infty} \frac{1}{|\Phi_N|}\sum_{n \in \Phi_N}\int_{\A_\ell(X)} (1\otimes f_1 \otimes \cdots \otimes f_\ell) \cdot T_\Delta^n (f_1 \otimes \cdots \otimes f_\ell \otimes \1 ) \diff \tilde \rho_a \\
        &\geq -3\delta/4 + \liminf_{N \to \infty} \frac{1}{|\Phi_N|} \sum_{n \in \Phi_N}  \int_{\HP_\ell(Z)}(1\otimes g_1 \otimes \cdots \otimes g_\ell) \cdot T_\Delta^n (g_1 \otimes \cdots \otimes g_\ell \otimes \1 )\diff \rho_{\pi(a)}  \geq \delta/4
    \end{align*}
    concluding the proof. 
\end{proof}

We end this section with the following result that is not used for the combinatorial applications, but it gives more information about the dynamics.
\begin{proposition}
    Let $\Xmt$ be an ergodic system with topological pronilfactor, then the map $X \to \cM(\A_\ell(X))$ given by $x \mapsto \tilde \rho_x$ is continuous. 
\end{proposition}

A similar proof can be found in \cite[Theorem 6.5]{kmrr1} 

\begin{proof}
    We denote $Z=Z_{\ell-1}(\mu)$ and $\pi = \pi_{\ell-1}$ the continuous factor map.  For $p= 2\ell$ we first prove that the function $\Theta \colon L^p (Z, m)^\ell \to L^\infty(Z,m)$ given by
    \begin{equation}
        \Theta(g)(z) = \int_{Z^\ell} \bigotimes_{i=1}^\ell g_i \diff \rho_z^* \quad \text{ for } g = g_1 \otimes \cdots \otimes g_\ell \in   L^p (Z, m)^\ell 
    \end{equation}
    is continuous. Notice that if $g, \tilde g \in L^p (Z, m)^\ell $ fulfill that there is $j \in \{1, \ldots, \ell \}$ such that $g_i = \tilde g_i$ for all $i \neq j$ then by Holder's inequality
    \begin{align*}
        \left| \Theta(g)(z) - \Theta(\tilde g)(z) \right| = \left| \int_{Z^\ell} (g_j - \tilde g_j)^{[j]}\cdot \prod_{i\neq j} g_i^{[i]} \diff \rho_z^* \right| \leq \norm{g_j - \tilde g_j}_{L^p(\rho_z^j)}  \prod_{i \neq j}  \norm{g_i}_{L^p(\rho_z^i)}.
    \end{align*}
    By \cref{prop marginals}, for each $i =1, \ldots, \ell$, the projection $\rho^i_z \leq i m$ for all $z \in Z$. In particular $\norm{f}_{L^p(\rho^i_z)} \leq i^{1/p} \norm{f}_{L^p(m)}$ for all $f \in L^p(m)$. Thus
    \begin{equation*}
        \norm{\Theta(g) - \Theta(\tilde g)}_{L^\infty (m)} \leq (\ell!)^{1/p} \cdot \norm{g_j - \tilde g_j}_{L^p(m)} \cdot \prod_{i \neq j}  \norm{g_i}_{L^p(m)}.
    \end{equation*}
    Using the previous inequality we deduce that the function $\Theta$ is continuous. 

    To conclude the proposition, we need to prove that for all continuous functions $f_0,f_1, \ldots, f_\ell \in C(X)$, $x \mapsto \int f_0 \otimes f_1 \otimes \cdots \otimes f_\ell \diff \tilde \rho_x$ is continuous. Since the zero coordinate is the Dirac delta at $x$, it suffices to show that
    \begin{equation*}
        x \mapsto \int_{Z^l} \E (f_1 \mid Z )\otimes \cdots \otimes \E(f_\ell \mid Z) \diff  \rho_{\pi(x)}^*
    \end{equation*}
    is continuous for all $f_1, \ldots, f_\ell \in C(X)$. Notice that for $f_1, \ldots, f_\ell \in C(X)$, $\E(f_i \mid Z) \in L^\infty (m) \subset L^p(m)$ for all $i =1, \ldots, \ell$ and that $\pi\colon X \to Z$ is continuous, so we only need to check that if $g \in L^p(m)^\ell$ then the function $\Theta(g)$ is continuous. 

    To prove this, since $z \mapsto \rho_z$ is a continuous ergodic decomposition, we get that if $g \in C(Z)^\ell$ then $z \mapsto \Theta(g)(z) = \int_{Z^{\ell+1}} \1 \otimes g \diff \rho_z $ is continuous. Since $C(Z)^\ell$ is dense in $L^p(m)^\ell$ and that $C(Z)$ is closed in $L^\infty(m)$, by the continuity of $\Theta \colon L^p(m)^\ell \to L^\infty(m)$, we conclude that for all $g \in L^p(m)^\ell $, $\Theta(g) \in C(Z)$, concluding the proof. 
\end{proof}

\section{Shifted primes progressive measures}
\subsection{Properties from progressive measures}

To deal with shifted primes is easier to work with totally ergodic measures. We introduce the family of measures $\rho^\circ_x$ that play a similar role to the measures $\rho_x$.  

\begin{proposition} \label{prop kappa rho_x}
    Let $\ell \geq 2$, $\Xmt$ be an ergodic $s$-step nilsystem, $x \mapsto \rho_x$ the continuous $\tilde T$-ergodic decomposition of the Haar measure $m_{\HP_\ell(X)}$. Let $\kappa \in \N$ be the number of connected components of $\supp \rho_x$ for any $x \in X$. If $w \mapsto \rho^\circ_w$ denotes the continuous $\tilde T^\kappa$-ergodic decomposition of $m_{\HP_\ell(X)}$ and $\rho^\circ_x = \rho^\circ_{(x, \ldots, x)}$ then 
    \begin{enumerate}[i)]
        \item $\rho_x = \frac{1}{\kappa} \sum^{\kappa-1}_{i=0} \tilde T^i \rho^\circ_x$ for all $x \in X$
        \item $\supp \rho^\circ_x$ is connected for all $x \in X$
        \item for $\mu$ almost every $x \in X$, $\rho^\circ_x$ is $\tilde T^\kappa$-totally ergodic.
    \end{enumerate}
\end{proposition}

\begin{remark*}
    The fact that $\kappa$ is well defined and that it does not depend on $x$ is ensured by \cref{prop connected components kappa}. 
\end{remark*}

\begin{proof}

    Recall that $\rho_w$ is defined for all $w \in \HP_\ell (X)$ and that by \cref{thrm new cont ergodic decomp} we have that for $x \in X$, $\rho_w = \rho_x$ if $w_0 = x$. Consider $w \mapsto \nu_w =\frac{1}{\kappa} \sum^{\kappa-1}_{i=0} \tilde T^i \rho^\circ_w$ then it is clear that $\nu_w$ defines a continuous $\tilde T$-ergodic decomposition. Since this is unique we get that $\rho_w = \nu_w$ for all $w \in \HP_\ell (X)$ and therefore $\rho_x = \nu_{(x, \ldots, x)}$ for all $x \in X$, concluding the first point. 

    By assumption, for all $x \in X$, $\supp \rho_x$ has $\kappa$ connected components $C^x_1, \ldots, C_{\kappa}^x$. 
    Take $x \in X$ and notice that each connected components are $\tilde T^\kappa$-invariant. So, on the one hand, $\supp \rho_x^\circ \subset C_j^x$ for some $j = 1, \ldots, \kappa $. On the other hand, by the previous part we have that,
    \begin{equation*}
        \supp \rho_x =  \bigsqcup_{i=0}^{\kappa-1} \tilde T^i \supp \rho_x ^\circ \subset  \bigsqcup_{i=1}^{\kappa} C_j^x = \supp \rho_x
    \end{equation*}
   which forces $\supp \rho_x^\circ = C_j^x$, so $\supp \rho_x^\circ$ is connected for for all $x \in X$.  

     We conclude the last part by noticing that for $\mu$-almost every $x \in X$, $\rho_x^\circ$ is $\tilde T^\kappa$-ergodic and $\supp \rho_x^\circ$ is connected, hence it is also $\tilde T^\kappa$-totally ergodic.  
\end{proof}


\begin{proposition}\label{uniform-Szemeredi} Let $\ell \in \N$. For every $\eta>0$, there exists $\delta=\delta(\ell,\eta)>0$ satisfying that: 

For every ergodic $s$-step nilsystem $(X,\mu,T)$, $x \mapsto \rho_x$ the continuous $\tilde T$-ergodic decomposition of $m_{\HP_\ell(X)}$, $f_1,\ldots,f_\ell\in C(X)$ non-negative functions bounded by $1$ and $\kappa \in \N$ as in \cref{prop kappa rho_x}, for all  $W\in \N$ and $x \in X$, if   
$$ \int_{\HP_\ell(X)} \1 \otimes  f_1 \otimes \cdots \otimes f_\ell  \diff \rho_x^\circ>\eta, $$
    then
    \begin{equation}
        \label{eq W progresive}
        \liminf_{N\to \infty} \frac{1}{N} \sum_{n\leq N} \int_{\HP_\ell(X)} (\1 \otimes  f_1 \otimes \cdots \otimes f_\ell )  \cdot T_{\Delta}^{Wn} ( f_1 \otimes \cdots \otimes f_\ell \otimes \1)\diff \rho^\circ_x \geq \delta/\kappa . 
    \end{equation} 
\end{proposition}

\begin{proof}
Let $W\in \N$ and $x\in X$ be fixed. Let $Y$ be the connected component
of $X$ containing $x$, and let $\mu_Y$ be its Haar
measure. Thus, the system $(Y,\mu_Y,T^{W\kappa})$ is totally ergodic. 

Denote by $\rho_y^S$ the measure defined in \cref{first def rho} for $(Y,\mu_Y,T^{W\kappa})$. By \cref{prop kappa rho_x}, for $\mu_Y$-almost every $y\in Y$,
the measure $\rho_y^\circ$ is $\tilde T^{W\kappa}$-ergodic.  Since $(y,\ldots,y)\in\operatorname{supp}\rho_y^\circ$,
unique ergodicity on its support gives
$\rho_y^\circ=\rho_y^S$
for $\mu_Y$-almost every $y\in Y$. Hence, by continuity we conclude that $\rho_y^\circ=\rho_y^S$ for every $y\in Y$.  

This way, we can use \cref{thrm rho is unif progressive} with $(Y,\mu_{Y},T^{W\kappa} ) $ and $\rho^\circ_x$ to get the desired conclusion that 
        $$\liminf_{N\to \infty} \frac{1}{N} \sum_{n\leq N} \int_{\HP_\ell(X)} (\1 \otimes  f_1 \otimes \cdots \otimes f_\ell )  \cdot T_{\Delta}^{W \kappa n} ( f_1 \otimes \cdots \otimes f_\ell \otimes \1)\diff \rho^\circ_x \geq \delta. $$ 
        This last inequality implies \eqref{eq W progresive}.  
\end{proof}

\subsection{Properties of the von Mangoldt function}
In what follows, we use the following notation: For $N\in\N$, we
write $[N]=\{1,\ldots,N\}$. For $a,b\in\R$ with $a\leq b$,
we denote the discrete interval $\{n\in\N:a\leq n\leq b\}$
by $[a,b]$. For functions $f,g:\N\to\C$, we write $f(n)=O(g(n))$, or
equivalently $f(n)\ll g(n)$, if there exists a constant $C>0$
such that $|f(n)|\leq C|g(n)|$ for all sufficiently large $n$.
Subscripts, as in $f\ll_{s_1,\ldots,s_k}g$, indicate the dependency of the constant $C$ on given parameters $s_1,\ldots,s_k$. We write $f=o_n(g)$ if $f(n)/g(n)\to 0$ as $n\to\infty$,
assuming that $g(n)\neq 0$ for all sufficiently large $n$.

We  use the \emph{Gowers uniformity norms}, which were introduced in \cite{gowers2001newszemeredi} and we define now. 
\begin{definition}[Local Gowers uniformity norm]
Let $d\in \N$ and $N\in \N$. Set
$$ \mathcal{C}_d(N)= \left\{
(x,h_1,\dots,h_d)\in\mathbb{Z}^{d+1}
:
x+\omega\cdot h\in[N]
\text{ for every }\omega\in \{0,1\}^d
\right\}.$$
For a function $F:\N\to \C$ we define
\begin{equation}\label{eq:gowers-local}
   \|F\|_{U^d[N]}^{2^d} =\frac{1}{|\mathcal{C}_d(N)|}\sum_{(x,h)\in\mathcal{C}_d(N)}
\prod_{\omega\in\{0,1\}^d}
\mathcal{C}^{|\omega|}
F(x+\omega\cdot h),
\end{equation}
where $\mathcal{C}$ denotes complex conjugation.
\end{definition}

The \text{von Mangoldt function} $\Lambda:\Z\to \R$ is defined by $\Lambda(n)=\log{p}$ if $n=p^m$ for some $m\in \N$ and $p\in \mathbb{P}$, and 0 otherwise. One could alternatively define 
$$ \Lambda'(n)=1_{\mathbb{P}}(n)\Lambda(n)$$
for $n\in \N$. In this article the role of $\Lambda$ and $\Lambda'$ are interchangeable by the classical estimate:
$$ \frac{1}{N} \sum_{n=1}^N |\Lambda(n)-\Lambda'(n)|= O(N^{-1/2}). $$

The following is a classical lemma. 
\begin{lemma}[{\cite[Lemma 1]{frantzikinakis_host2023multiple}}]\label{prime-to-vm}
    If $|a_n|\leq 1$ for $n\in \N$, then
    $$\left| \frac{1}{\pi(N)} \sum_{p\in \mathbb{P},p\leq N} a(p) - \frac{1}{N} \sum_{n=1}^N \Lambda'(n) a(n) \right| =o_N(1). $$
\end{lemma}

Given a $w\in \N$ and $r\in \Z$, if we set
$$W=W(w)=\prod_{p\in \mathbb{P},p<w}p, $$
then for $n\in N$ we define 
$$ \Lambda'_{w,r}(n) = \frac{\phi(W)}{W} \cdot \Lambda'(Wn+r),$$
where $\phi$ is the Euler function. For $M\in \N$ denote $R(M)=\{r\in [M]: (r,M)=1\}$. We will crucially use the following results. 
\begin{lemma}[{ \cite[Lemma 3.5]{Frantzikinakis_Host_Kra13}}]\label{Gowers-Control}
Let $\ell \in \mathbb{N},(X, \mathcal{X}, \mu)$ be a probability space, $T_1, \ldots, T_{\ell}: X \rightarrow X$ be commuting invertible measure preserving transformations, $f_1, \ldots, f_\ell \in L^{\infty}(\mu)$ be functions bounded by 1 , and $c_{1},\ldots, c_\ell\in \Z$ be constants. Let $a: \mathbb{N} \rightarrow \mathbb{C}$ be a sequence of complex numbers satisfying $a(n) / n^c \rightarrow 0$ for every $c>0$. Then there exists $d \in \mathbb{N}$, depending only on $\ell$ , such that
$$
\left\|\frac{1}{N} \sum_{n=1}^N a(n) \cdot   \prod_{i=1}^{\ell}T_i^{c_{  i}n}  f_i  \right\|_{L^2(\mu)} 
\ll_d\left\|a \right\|_{U_d\left([N]\right)}+o_N(1) .
$$

Furthermore, the implicit constant is independent of the sequence $(a(n))_{n \in \mathbb{N}}$, and the $o_N(1)$ term depends only on the integer $d$ and on the sequence $(a(n))_{n \in \mathbb{N}}$.
\end{lemma}
 \begin{remark}
     The original statement in \cite{Frantzikinakis_Host_Kra13} has \cref{Gowers-Control} with $\left\|a  \ind{[N]}\right\|_{U_d\left(\Z_{dN}\right)}$ in place of $\left\|a \right\|_{U_d\left([N]\right)}$. Our version follows from the fact that these quantities differ only by a constant depending solely in $d$.
 \end{remark}
\begin{definition}
    If $X=G/\Gamma$ is an $s$-step nilmanifold, $f\in C(X)$, and $g\in G$, then the sequence $(f(g^ne_X))_{n\in \N}$ is an $s$-step nilsequence. In addition, a $0$-step nilsequence is a constant sequence. 
\end{definition}

\begin{proposition}[{\cite[Proposition 10.1]{Green_Tao10} and \cite[Theorem 1.3]{Green_Tao_Ziegler12}}]\label{Green-Tao-inverse-theorem}
  For any $0<\delta \leqslant 1$ and any $C \geqslant 20$, there exists a finite collection $\mathcal{M}_{s, \delta, C}$ of nilmanifolds $G / \Gamma=\left(G / \Gamma, d_{G / \Gamma}\right)$ with the following property. Let $N \geqslant 1$. Suppose that $N^{\prime} \in[C N, 2 C N]$ is a prime, that $v: \mathbb{Z}_{N^{\prime}} \rightarrow \mathbb{R}^{+}$is an $(s+2) 2^{s+1}$ pseudorandom measure\footnote{See \cite[Definitions 6.2 and 6.3]{Green_Tao10}}, that $f:[N] \rightarrow \mathbb{R}$ is a function with $|f(n)| \leqslant v(n)$ for all $n \in[N]$ and that $\|f\|_{U^{s+1}[N]} \geqslant \delta$. Then there exists $G / \Gamma \in \mathcal{M}_{s, \delta, C}$ together with a 1 -bounded $s$-step nilsequence $\left(F\left(g^n x\right)\right)_{n \in \mathbb{Z}}$ with Lipschitz constant $O_{s, \delta, C}(1)$, such that
$$
\left|\mathbb{E}_{n \leqslant N} f(n) F\left(g^n x\right)\right| \gg_{s, C, \delta} 1 .
$$  
\end{proposition}
\cref{Green-Tao-inverse-theorem} is particularly useful for averages including the von Mangoldt function because of the following proposition.
\begin{proposition}[{See \cite[Proposition 6.4]{Green_Tao10}}]\label{vmf-bounded-by-srm}
Let $D>1$ be arbitrary. Then there is a constant $C_0=C_0(D)$ such that the following is true. Let $C \geqslant C_0$, and suppose that $N^{\prime} \in[C N, 2 C N]$. Let $r\in R(W)$. Then there exists a $D$-pseudorandom measure $\nu: \mathbb{Z}_{N^{\prime}} \rightarrow \mathbb{R}^{+}$ which obeys the pointwise bounds
$$
1+\Lambda_{w,r}^{\prime}(n)  \ll_{D, C} \nu(n)
$$
for all $n \in\left[N^{3 / 5}, N\right]$, where we identify $n$ with an element of $\mathbb{Z}_{N^{\prime}}$ in the obvious manner.
\end{proposition}

For the next proposition, we follow the notation from \cite{Bergelson_Leibman_Son14}. We  denote by $\mathcal{N}_{s,d}$ the universal free nilmanifold of nilpotency class $s$ with $d$ continuous and $d$ discrete generators. This nilmanifold satisfies that any nilmanifold of nilpotency class $\leq s$ and with $\leq d$ continuous and discrete generators respectively, is a factor of $\mathcal{N}_{s,d}$. Given $s,d\in \N$ and $M>0$, we  denote by $\mathcal{L}_{s,d,M}$ the set of basic nilsequences $f(g^ne_X)$ where $f\in C(\mathcal{N}_{s,d})$ is Lipschitz with constant $M$ and $|f|\leq M$. We remark that a smooth metric on each nilmanifold $\mathcal{N}_{s,d}$ is assumed to be chosen. The following proposition is a softer version of \cite[Proposition 10.2]{Green_Tao10} and \cite{Green_Tao12Mobius}, stated by Bergelson, Leibman, and Son \cite[Proposition 7.3]{Bergelson_Leibman_Son14} which is particularly useful to us as we need the $w$ to not depend on $N$. 
\begin{proposition}\label{uniformity-vmf}
 For any $s, d \in \mathbb{N}$ and $M>0$, the quantity
$$
 \sup _{\substack{\eta \in \mathcal{L}_{s, d, M} \\ r \in R(W(w))}}\left|\frac{1}{N} \sum_{n=1}^N\left(\Lambda_{w, r}^{\prime}(n)-1\right) \eta(n)\right|
$$
converges to $0$ as $N\to \infty$ and then $w\to \infty$. 
\end{proposition}

We are now in position of proving the following lemma.
 \begin{lemma}\label{pre-strong-antiuniformity}
Let $d,s,M_0,m\in \N$. Then, we have that 
$$\sup_{\psi\in \mathcal{L}_{s,m,M_0}} \norm{(\Lambda_{w,1}'-1)\psi}_{U^d([N])} $$
converges to $0$ as $N\to \infty$ and then $w\to \infty$.  
 \end{lemma}
\begin{proof}
After dividing by a constant depending on $\psi $, we may assume that the supremum is taken over nilsequences bounded by $1$.  

 Let us assume by contradiction that there are $\delta>0$, an increasing sequences $(w_i)_{i\in \N}$ and $(N_{i,j})_{j\in \N}$, and $\psi_{i,j}\in \mathcal{L}_{s,m,M_0}$ such that for each $i,j\in \N$,
$$\norm{(\Lambda_{w_i,1}'-1)\psi_{i,j} }_{U^d([N_{i,j}])}\geq \delta. $$

Fix $i\in \N$ from now on. By passing to a subsequence of $(N_{i,j})_{j\in \N}$, we can assume that for all $j\in \N$ 
$$\norm{\ind{[N_{i,j}^{3/5},N_{i,j}]}(\Lambda_{w_i,1}'-1)\psi_{i,j} }_{U^d([N_{i,j}])}\geq \delta/2,  $$
where we used that
\begin{align*}
        \norm{\ind{[N_{i,j}^{3/5}]}(\Lambda_{w_i,1}'-1)\psi_{i,j} }_{U^d([N_{i,j}])}&\leq  \norm{\ind{[N_{i,j}^{3/5}]} (\log{(N_{i,j})} +1)}_{U^d([N_{i,j}])}\\
        &= (\log{(N_{i,j})} +1)  \norm{\ind{[N_{i,j}^{3/5}]} }_{U^d([N_{i,j}])}\\
        &\leq (\log{(N_{i,j})} +1) \left(N_{i,j}^{-2/5}\right)^{(d+1)/2^d}.
\end{align*}

Then, \cref{vmf-bounded-by-srm} gives that for $D=(d+1)2^{d}$, there is $C=C(d)$ such that for each $N_{i,j}'\in [CN_{i,j},2CN_{i,j}]$ prime there is a $D$-pseudorandom measure $\nu:\Z_{N_{i,j}'}\to \R_+$ such that 
$$|\ind{[N_{i,j}^{3/5},N_{i,j}]}(n)(\Lambda_{w_i,1}'(n)-1) \psi_{i,j}|\leq |\ind{[N_{i,j}^{3/5},N_{i,j}]}(n)(\Lambda_{w_i,1}'(n)-1) |\ll_{d} \nu(n), $$
for each $n\in [N_{i,j}]$. Now we can use \cref{Green-Tao-inverse-theorem} to find a nilmanifold $G_{i,j}/\Gamma_{i,j}\in \mathcal{M}_{d-1,\delta/2,C}$ and $1$-bounded $(d-1)$-step nilsequence $\Psi_{i,j}$ of whose implicit Lipschitz constant is $M_1=M_1(d,\delta)$ such that for each $j\in\N$, 
$$\left| \frac{1}{N_{i,j}} \sum_{n\leq N_{i,j}} \ind{[N_{i,j}^{3/5},N_{i,j}]} (\Lambda_{w_i,1}'  -1) \psi_{i,j}  \Psi_{i,j} \right|\geq c' $$
for some $c'=c(d,\delta)'>0$. In particular, after passing to another subsequence of $(N_{i,j})_{j\in \N}$, we can assume that  
$$\left| \frac{1}{N_{i,j}} \sum_{n\leq N_{i,j}} (\Lambda_{w_i,1}'-1)   \psi_{i,j}  \Psi_{i,j} \right|\geq c $$
where $c=c'/2$.

The product $\psi_{i,j} \Psi_{i,j}$ is a 
$M=(M_0+M_1)$-Lipschitz nilsequence, over the product nilmanifold, which is an $(s+d)$-nilmanifold, with $h$ generators where $h=h(m,d,\delta)$ depends on $m,d,\delta$. Therefore, we have that for each $j$:
\begin{equation}\label{eq-vmf-against-nil}
    c\leq \left| \frac{1}{N_{i,j}} \sum_{n\leq N_{i,j}}(\Lambda_{w_i,1}'-1)   \psi  \Psi_{i,j} \right| \leq   \sup _{\substack{\eta \in \mathcal{L}_{s+d, h, M} \\ r \in R(w_i)}}\left|\frac{1}{N_{i,j}} \sum_{n=1}^{N_{i,j}}\left(\Lambda_{w_i, r}^{\prime}(n)-1\right) \eta(n)\right|.
\end{equation}
Nevertheless, by \eqref{uniformity-vmf} if we take $j\to \infty$ and then $i\to \infty$, the right-hand side of \eqref{eq-vmf-against-nil} goes to zero, which is a contradiction. Thus, the result follows.
\end{proof}
 \begin{corollary}\label{strong-antiuniformity}
 For every basic nilsequence $\psi_w(n)=f(g_w^n e_X)$ where $X=G/\Gamma$ is a nilmanifold, $f\in C(X)$ and $g_w\in G$ depends on $w$, we have that 
$$ \norm{(\Lambda_{w,1}'-1)\psi_w}_{U^d([N])}$$
converges to $0$ as $N\to \infty$ and then $w\to \infty$.
 \end{corollary}
\begin{proof}
Let $X=G/\Gamma$ be an $s$-step nilmanifold on $m$ generators. Let $\psi_w(n) = f(g_w^n  e_X)$ be a nilsequence for $f\in C(X)$ and $g_w\in G$. For $\varepsilon>0$, take $F\in C(X)$ an $M_0$-Lipschitz function for some $M_0>0$ such that $\norm{f-F}_\infty\leq \varepsilon$. Then, we have that
\begin{align*}
&\norm{(\Lambda_{w,1}'-1)\psi_w}_{U^d([N])} \\ \leq& 
\norm{(\Lambda_{w,1}'-1)F(g_w^n e_X)}_{U^d([N])}  + \norm{(\Lambda_{w,1}'-1)(f(g_w^n e_X)-F(g_w^n e_X)) }_{U^d([N])} \\
\leq& \sup_{\eta \in \mathcal{L}_{s,m,M_0}}  \norm{(\Lambda_{w,1}'-1)\eta}_{U^d([N])}  + \norm{(\Lambda_{w,1}'-1)(f(g_w^n e_X)-F(g_w^n e_X)) }_{U^d([N])} .
\end{align*}
The first term in the last expressions converges to $0$ as $N\to \infty$ and then $w\to \infty$ by \cref{pre-strong-antiuniformity}. For the second term, we notice that:
\begin{align*}
 &\left|\norm{(\Lambda_{w,1}'-1)(f(g_w^n e_X)-F(g_w^n e_X)) }_{U^d([N])} \right|^{2^d}\\ &= \frac{1}{|\mathcal{C}_d(N)|}
\sum_{(n,h)\in\mathcal{C}_d(N)}
\prod_{\omega\in\{0,1\}^d}
\mathcal{C}^{|\omega|} \left(
(\Lambda_{w,1}'(n+\omega\cdot h)-1)(f(g_w^{(n+\omega\cdot h)} e_X)-F(g_w^{(n+\omega\cdot h)} e_X)) \right)\\
&\leq \varepsilon^{2^d }\frac{1}{|\mathcal{C}_d(N)|}
\sum_{(n,h)\in\mathcal{C}_d(N)}
\prod_{\omega\in\{0,1\}^d}
(\Lambda_{w,1}'(n+\omega\cdot h)+1) \ll_d \varepsilon^{2^d},
\end{align*}
where in the last inequality we used that 
$$\frac{1}{|\mathcal{C}_d(N)|}
\sum_{(n,h)\in\mathcal{C}_d(N)}
\prod_{\omega\in\{0,1\}^d}
(\Lambda_{w,1}'(n+\omega\cdot h)+1)\ll_d 1 $$
for this, see for example \cite[Theorem 1.6]{TaoTeravainen}. Thus, after taking $N\to \infty$, then $w\to \infty$, we can take $\epsilon\searrow 0$ to conclude that this last quantity goes to $0$, finishing the proof of the lemma.
\end{proof}

\subsection{The proof of the main result}
Now we prove the main result of this section. Given $k\in \N$, for a function $g:X\to \C$ in a set $X$, we  denote $g^{[i]}:X\to \C^k$ the function $1\otimes \cdots \otimes 1 \otimes g \otimes 1 \cdots \otimes 1 $ where $g$ is in the $i$th-coordinate. In what follows we prove that $\rho^\circ_x$ is $(\P-1)$-uniformly progressive. 
\begin{theorem}\label{uniformity-measures-in-primes}
    Let $\ell\geq 2$. Let $\kappa,s,d\in \N$ For every ergodic $s$-step nilsystem $(X,\mu,T)$ in $d$ generators, $\kappa$ is as in \cref{prop kappa rho_x}, and $f_1,\ldots,f_\ell\in C(X)$ nonnegative functions bounded by $1$, we have that for every $\eta>0$, there exists a constant 
    $$\delta=\delta(\ell,\kappa,\eta,s,d,f_1)>0$$
    such that if for some $x \in X$ 
    $$ \int_{\HP_\ell(X)} \1 \otimes  f_1 \otimes \cdots \otimes f_\ell  \diff \rho_x^\circ>\eta, $$
    then
    $$\liminf_{N\to \infty} \frac{1}{\pi(N)} \sum_{p\in \mathbb{P},p<N} \int_{\HP_\ell(X)} (\1 \otimes  f_1 \otimes \cdots \otimes f_\ell) \cdot T_\Delta^{p-1} (f_1 \otimes \cdots \otimes f_\ell \otimes \1) \diff \rho_x^\circ \geq \delta.  $$
\end{theorem}
We remark that our proof is similar to the proof of \cite[Theorem 4]{Frantzikinakis_Host_Kra07}.
\begin{proof}
    By definition we have that $\rho^\circ_x$ is $\tilde{T}^\kappa$-invariant. By \cref{prime-to-vm}, it is enough to prove the following claim
    
    \underline{Claim:} For $w$ big enough
    \begin{equation}\label{lower-bound-along-primes}
      \liminf_{N\to \infty} \frac{1}{N} \sum_{n\leq N} \Lambda'(n+1) \int_{\HP_\ell(X)} (\1 \otimes  f_1 \otimes \cdots \otimes f_\ell) \cdot T_\Delta^{n} (f_1 \otimes \cdots \otimes f_\ell \otimes \1) \diff \rho_x^\circ >\delta/W.  
    \end{equation}
For this purposes, it would be enough to prove that for $w$ big enough, we have that 
    \begin{equation*}
      \liminf_{N\to \infty} \frac{1}{N} \sum_{n\leq N} \Lambda'_{w,1}(n) \int_{\HP_\ell(X)} (\1 \otimes  f_1 \otimes \cdots \otimes f_\ell) \cdot T_\Delta^{Wn} (f_1 \otimes \cdots \otimes f_\ell \otimes \1) \diff \rho_x^\circ >\delta/W.  
    \end{equation*}
    for some $\delta$ to be determined. We claim that
    \begin{equation}\label{eq-1}
        \limsup_{N\to \infty} \left|\frac{1}{N} \sum_{n\leq N} (\Lambda'_{w,1}(n)-1) \int_{\HP_\ell(X)} (\1 \otimes  f_1 \otimes \cdots \otimes f_\ell) \cdot T_\Delta^{Wn} (f_1 \otimes \cdots \otimes f_\ell \otimes \1) \diff \rho_x^\circ  \right|
    \end{equation}
    goes to $0$ as $w\to \infty$.
    Indeed, to show this, denote $R=(\ell \kappa )!$ and $$a(n)=(\Lambda_{w,1}'(n)-1) f_1(T^{Wn}x)$$ for each $n\in \N$. Observe that in the limsup in \eqref{eq-1} we may restrict to cutoffs divisible by $R$, since changing the cutoff by
fewer than $R$ terms costs only $O_{w,R}(\log N/N)$.
Writing the cutoff as $RN$, the average can be written,
up to an error $o_N(1)$, as
    \begin{align*}
          &  \frac{1}{R}\sum_{j=0}^{R-1} \frac{1}{N} \sum_{n\leq N}  a(Rn+j) \int_{\HP_\ell(X)} (\1 \otimes  f_1 \otimes \cdots \otimes f_\ell) \cdot T_\Delta^{WRn+Wj} (\1 \otimes f_2 \otimes \cdots \otimes f_\ell \otimes \1) \diff \rho_x^\circ \\
          &= \frac{1}{R} \sum_{j=0}^{R-1} \int_{\HP_\ell(X)} (\1 \otimes  f_1 \otimes \cdots \otimes f_\ell) \left[ \frac{1}{N} \sum_{n=1}^N      a(Rn+j) \Big( \prod_{i=1}^{\ell-1} \tilde{T}^{(WR/i )n}(T^{ Wj}f_{i+1})^{[i]} \Big) \right]   \diff \rho_x^\circ.
          \end{align*}
 Dropping the error, by Cauchy-Schwarz the last is upper bounded by
    $$\frac{1}{R} \sum_{j=0}^{R-1}  \norm{\frac{1}{N} \sum_{n=1}^N  a(Rn+j) \prod_{i=1}^{\ell-1} \tilde{T}^{(WR/i )n} (T^{ Wj}f_{i+1})^{[i]}}_{L^2(\rho_x^\circ)} .$$

    By \cref{Gowers-Control} the last expression is $\ll_D$ than
    $$ o_{N\to \infty}(1) +\frac{1}{R} \sum_{j=0}^{R-1} \norm{a(R\cdot +j) }_{U^d([N])} , $$
    where $D$ only depends on $\ell$. We notice that for $n\in \N$, 
    $$a(Rn+j)= a(m) 1_{R\Z+ j}(m)= (\Lambda_{w,1}'(m)-1) f_1(T^{Wm}x)1_{R\Z+ j}(m)  $$
    for some $m\in \N$. Call $\psi_{j,W}(m)=f_1(T^{mW}x)1_{R\Z+ j}(m).  $
    We have that $\psi_{j,W}$ is a basic nilsequence. We thus have that 
   $$\limsup_{N\to \infty} \norm{a(R\cdot +j) 1_{[N]}}_{U^D([N])} \ll_{R,d}
   \limsup_{N\to \infty}\norm{(\Lambda_{w,1}'-1) \psi _{j,W}}_{U^D([RN+j])}.  $$
Finally, by \cref{strong-antiuniformity}, the right hand side in the previous equation   goes to $0$ as $w\to \infty$ concluding the claim. 

    To conclude the theorem, we observe that by \cref{uniform-Szemeredi} we have that 
    $$\liminf_{N\to \infty} \frac{1}{N} \sum_{n\leq N} \int_{\HP_\ell(X)} (\1 \otimes  f_1 \otimes \cdots \otimes f_\ell )  \cdot T_{\Delta}^{Wn} ( f_1 \otimes \cdots \otimes f_\ell \otimes \1)\diff \rho^\circ_x \geq \delta  $$
where $\delta$ only depends on $\eta,\ell ,\kappa$. Using this and the claim, we have that for $w$ big enough: 
\begin{align*}
  &\liminf_{N\to \infty} \frac{1}{N} \sum_{n\leq N} \Lambda_{w,1}'(n) \int_{\HP_\ell(X)} (\1 \otimes  f_1 \otimes \cdots \otimes f_\ell )  \cdot T_{\Delta}^{Wn} ( f_1 \otimes \cdots \otimes f_\ell \otimes \1)\diff \rho^\circ_x\\
  \geq& -\delta/2 + \lim_{N\to \infty} \frac{1}{N} \sum_{n\leq N} \int_{\HP_\ell(X)} (\1 \otimes  f_1 \otimes \cdots \otimes f_\ell )  \cdot T_{\Delta}^{Wn} ( f_1 \otimes \cdots \otimes f_\ell \otimes \1)\diff \rho^\circ_x \geq \delta/2. 
\end{align*}
We conclude the proof by observing that $\delta$ depends only in $\eta,\ell,\kappa$. However, the final bound that we get in \eqref{lower-bound-along-primes} depends in $w$, which depends on $s,d,f_1$ by \cref{strong-antiuniformity}. Thus, the $\tilde{\delta}=\delta/(2W(w))$ of the statement of the theorem depends on $\ell,\kappa,\eta,s,d,f_1$; concluding. 
\end{proof}

    \section{Dynamical consequences in nilsystems} \label{sec top consequences}

    In this section, we provide further consequences for points in pronilsystems. The main result is \cref{theorem summary} which extends \cite[Theorem 4.6]{radic2026Uniformity}.

\begin{proposition} \label{cor supports}
    Let $(Z, m ,T)$ be an ergodic pronilsystem, $k \leq \ell$ fixed natural numbers and $Z_{k-1}$ its $(k-1)$-step pronilfactor. Denote $\rho_z \in \cM ( \HP_\ell (Z))$ and $\eta_u \in \cM ( \HP_\ell (Z_{k-1}))$ the continuous ergodic decomposition of, respectively, $m_{HP_\ell(Z)}$ and $m_{HP_\ell(Z_{k-1})}$ from \cref{first def rho}. For all $L = \{ \ell_1, \ldots, \ell_k \} \subset \{1, \ldots, \ell\}$  and $z \in Z$,
    \begin{equation*} 
        \supp \rho_z^L = (\pi_{k-1} \times \cdots \times \pi_{k-1})^{-1} (\supp \eta_{\pi_{k-1}(z)}^L)
    \end{equation*}
\end{proposition}

\begin{proof}
    By \eqref{eq good projection}, it is clear that $\supp \rho_z^L \subset (\pi_{k-1} \times \cdots \times \pi_{k-1})^{-1} (\supp \eta_{\pi_{k-1}(z)}^L)$. For the converse, by \cref{prop cont k-step max pronilfactors}, if we denote the continuous disintegration of $m$ over $m_{k-1}$ by $u \mapsto \beta_u$, we also get that $\supp \beta_u = \pi_{k-1}^{-1}(u)$. Thus, for any $z \in Z$ and $ w \in  (\pi_{k-1} \times \cdots \times \pi_{k-1})^{-1} (\supp \eta_{\pi_{k-1}(z)}^L) $, if $f_1, \ldots, f_k \in C(Z)$ are nonnegative continuous functions with $f_i (w_i) >0$ for all $i = 1, \ldots, k$, then by fully supported in the fiber $\E(f_i \mid Z_{k-1})(\pi_{k-1}(w_i)) >0 $ for all $i=1, \ldots, k$. By continuity, there is a neighborhood $W$ of $w '=(\pi_{k-1}\times \cdots \times \pi_{k-1})(w)$ such that $\big(\bigotimes_{i=1}^k \E(f_i \mid Z_{k-1}) \big) (v) >\alpha$ for all $v \in W$ for some small constant $\alpha > 0$. Using again \eqref{eq good projection} we conclude that
    \begin{align*}  
        \int_{Z^k} f_1 \otimes \cdots \otimes f_k \diff \rho_z^L = \int_{Z^k_{k-1}} \E(f_1\mid Z_{k-1}) \otimes \cdots \otimes \E(f_k \mid Z_{k-1}) \diff \eta^L_{\pi_{k-1}(z)} \\
        \geq \int_{W} \E(f_1\mid Z_{k-1}) \otimes \cdots \otimes \E(f_k \mid Z_{k-1}) \diff \eta^L_{\pi_{k-1}(z)} \geq \alpha \eta^L_{\pi_{k-1}(z)}(W) >0, 
    \end{align*}
    where we know that $\eta^L_{\pi_{k-1}(z)}(W) >0$ because $w' \in \supp \eta^L_{\pi_{k-1}(z)}$. Since this is true for any non-negative continuous function $f_i \in C(Z)$ with $(f_1 \otimes \cdots \otimes f_k) (w) >0$ we conclude that $w \in \supp \rho_z^L$. 
\end{proof}

With this we recover the result of \cite[Theorem 4.3]{glasner_Huang_Shao_Weiss_Ye2025topological} for the special case of pronilsystem, but replacing the $G_\delta$ dense set $Z'$ by a full measure subset $Z'$.

\begin{corollary}
    Let $(Z, m ,T)$ be an ergodic pronilsystem, $k$ fixed natural numbers and $Z_{k-1}$ its $(k-1)$-step pronilfactor. There exists a set $Z' \subset Z$ of full measure such that for all distinct $\ell_1, \ldots, \ell_k \in \N$ and all $z \in Z'$
    \begin{equation}  \label{eq saturation}
        \overline{\{ (T^{\ell_1 n} z, \ldots, T^{\ell_k n} z) \colon n \in \Z\}  } = (\pi_{k-1}^{(k)})^{-1} \Big(\overline{\{ (T^{\ell_1 n} \pi_{k-1}(z), \ldots, T^{\ell_k n} \pi_{k-1}(z)) \colon n \in \Z\}  } \Big).
    \end{equation}
    where $\pi_{k-1}^{(k)}=\pi_{k-1} \times \cdots \times \pi_{k-1}$.
\end{corollary}

\begin{proof}
    Fixing $\ell \in \N$ arbitrary, for all $L = \{\ell_1, \ldots, \ell_k\} \subset \{1, \ldots, \ell\}$, by \cref{cor supports}, $\supp \rho_z^L = (\pi_{k-1} \times \cdots \times \pi_{k-1})^{-1} (\supp \eta_{\pi_{k-1}(z)}^L)$ and there exists $Z^{(\ell)}\subset Z$ with $m(Z^{(\ell)})=1$ such that $\rho_z = \sigma_z$ and $\eta_{\pi_{k-1}(z)} = \nu_{\pi_{k-1}(z)}$, where as before $\sigma_z$ and $\nu_u$ denote the unique invariant measures of $\overline{\cO_{\tilde T}(z, \ldots,z)}$ and $\overline{\cO_{\tilde T}(u, \ldots,u)}$. By unique ergodicity we get that $\supp \nu_{\pi_{k-1}(z)} = \overline{\cO_{\tilde T}(\pi_{k-1}(z), \ldots,\pi_{k-1}(z))} $ and $\supp \sigma_{z} = \overline{\cO_{\tilde T}(z, \ldots,z)} $, so projecting to $L = \{\ell_1, \ldots, \ell_k\}$ we get that all $z \in Z^{(\ell)}$ fulfill \eqref{eq saturation}. 
    
    This proves the statement for all subsets $L \subset \{1, \ldots , \ell\}$ of size $k$, so we conclude by taking $Z'= \bigcap_{\ell=k}^\infty Z^{(\ell)}$. 
\end{proof}

\begin{corollary} \label{cor contention of the image in the support}
    Let $(Z, m ,T)$ be an ergodic pronilsystem, $k \leq \ell$ fixed natural numbers and $Z_{k-1}$ its $(k-1)$-step pronilfactor. Denote $\rho_z \in \cM ( \HP_\ell (Z))$ and $\eta_u \in \cM ( \HP_\ell (Z_{k-1}))$ the continuous ergodic decomposition of, respectively, $m_{HP_\ell(Z)}$ and $m_{HP_\ell(Z_{k-1})}$ from \cref{first def rho}. For all $L = \{ \ell_1, \ldots, \ell_k \} \subset \{1, \ldots, \ell\}$  and $z \in Z$,
    \begin{equation*} 
         (\pi_{k-1} \times \cdots \times \pi_{k-1})^{-1} (\pi_{k-1}(z) , \ldots, \pi_{k-1}(z)) \subset \supp \rho_z^L 
    \end{equation*}
\end{corollary}

\begin{proof}
    By \cref{cor supports} we only need to show that $(\pi_{k-1}(z) , \ldots, \pi_{k-1}(z)) \in \supp \eta_{\pi_{k-1}(z)}^L$ which is a consequence of \cref{lemma ergodic fully supp decomposition}. 
\end{proof}

    \begin{proposition} \label{prop nbhds sumsets in pronil}
        Let $(Z, m, T)$ be an ergodic nilsystem, $z \in Z$ and $k \geq 2$.  For every $1 \leq \ell_1 < \cdots < \ell_k \leq \ell$, if $z_1, \ldots, z_k \in Z$ are points such that  $\pi_{k-1}(z) = \pi_{k-1}(z_i)$ and $V_i$ is a neighborhood of $z_i$ for $i=1, \ldots, k$, then there exists $B \subset \N$ infinite such that 
        \begin{equation*}
             \Big\{ \sum_{b \in I} b \colon I \subset B, |I| = \ell_i \Big\} \subset \{ n \in \N \mid T^n z \in V_i \} \quad \text{ for all } i=1, \ldots, k. 
        \end{equation*}
    \end{proposition}

    \begin{proof}
        This is direct from \cref{cor contention of the image in the support} and \cref{thrm rho is unif progressive}.
    \end{proof}

    In the case of nilsystem we have a similar consequence, but for shifted primes.

\begin{proposition} \label{cor contention of the image in the support primes}
    Let $(Z, m ,T)$ be an ergodic nilsystem, $2\leq k \leq \ell$ fixed natural numbers and $Z_{k-1}$ its $(k-1)$-step nilfactor. Denote $\rho_z^\circ \in \cM ( \HP_\ell (Z))$ as in \cref{prop kappa rho_x} and similarly for $\eta_u^\circ \in \cM ( \HP_\ell (Z_{k-1}))$. For all $L = \{ \ell_1, \ldots, \ell_k \} \subset \{1, \ldots, \ell\}$  and $z \in Z$,
    \begin{equation}  \label{eq fiber in rho circ}
         (\pi_{k-1} \times \cdots \times \pi_{k-1})^{-1} (\pi_{k-1}(z) , \ldots, \pi_{k-1}(z)) \subset \supp (\rho_z^\circ)^L 
    \end{equation}
    In particular, if $z \in Z$, $k \geq 2$ and $z_1, \ldots, z_k \in Z$ are points such that  $\pi_{k-1}(z) = \pi_{k-1}(z_i)$ and $V_i$ is a neighborhood of $z_i$ for $i=1, \ldots, k$, then there exists $B \subset \P-1$ infinite such that 
        \begin{equation} \label{eq nebhds and primes}
             \Big\{ \sum_{b \in I} b \colon I \subset B, |I| = \ell_i \Big\} \subset \{ n \in \N \mid T^n z \in V_i \} \quad \text{ for all } i=1, \ldots, k. 
        \end{equation}
\end{proposition}

\begin{proof}
    First notice that for all $z \in Z$, by \cref{prop kappa rho_x}, $\supp \rho_z^\circ$ is a connected component of $\supp \rho_z$ and hence $\supp (\rho_z^\circ)^L $ is a connected component of $ \supp \rho_z^L$ (since the coordinate projection is continuous and open
onto its image). Since $$(G_kz,\ldots,G_kz)= (\pi_{k-1} \times \cdots \times \pi_{k-1})^{-1} (\pi_{k-1}(z) , \ldots, \pi_{k-1}(z)) \subset Z^k$$ is connected because $G_k$ is connected, and $(z, \ldots, z) \in \supp (\rho_z^\circ)^L$ we get \eqref{eq fiber in rho circ} directly from \cref{cor contention of the image in the support}. We deduce \eqref{eq nebhds and primes} from the previous analysis, \cref{uniformity-measures-in-primes} and \cref{prime-EPs-implies-sumsets}.
\end{proof}

For a point $x \in X $ and an open set $V \subset X$, we denote $N(x,V) = \{n \in \N \colon T^n x \in V\}$. 

\begin{theorem} \label{theorem summary} 
    Let $k \geq 2$ be an integer, $(X,T)$  be a minimal pronilsystem. For $x,y \in X$, the following are equivalent
    \begin{enumerate}
        \item \label{summary theorem pt 0} $\pi_{k-1}(x) = \pi_{k-1}(y) $;
        \item \label{summary theorem pt 1} For every neighborhood $V$ of $y$ there exists $b_1, \ldots, b_k \in \N$ distinct natural numbers such that 
        $$\sum_{i \in I} b_i \in N(x,V) \quad \text{ for all } \quad I \subset \{1,\ldots, k\}, \ I \neq \emptyset; $$
        \item \label{summary theorem pt 2} For every neighborhood $V$ of $y$ there exists an infinite set $B \subset \N$ such that 
        $$ \Big\{ \sum_{b \in I} b \colon I \subset B, 1 \leq |I| \leq k  \Big\} \subset N(x,V);$$ 
    \item \label{summary theorem pt 3} For every neighborhood $V$ of $y$ and all integers $1 \leq\ell_1 < \cdots < \ell_k$, there exists an infinite set $B \subset \N$ such that
    $$ \Big\{ \sum_{b \in I} b \colon I \subset B, |I| = \ell_1, \ldots, \ell_k \Big\} \subset N(x,V);$$
        \item \label{summary theorem pt 4}  For every neighborhood $V$ of $y$ there exists $q_1, \ldots, q_k \in \P-1$ distinct natural numbers such that 
        $$\sum_{i \in I} q_i \in N(x,V) \quad \text{ for all } \quad I \subset \{1,\ldots, k\}, \ I \neq \emptyset; $$
        \item \label{summary theorem pt 5} For every neighborhood $V$ of $y$ there exists an infinite set $B \subset \P-1$ such that 
        $$\Big\{ \sum_{b \in I} b \colon I \subset B, 1 \leq |I| \leq k  \Big\}  \subset N(x,V);$$ 
    \item \label{summary theorem pt 6} For every neighborhood $V$ of $y$ and all integers $1 \leq \ell_1 < \cdots < \ell_k$, there exists an infinite set $B \subset \P-1$ such that
    $$\Big\{ \sum_{b \in I} b \colon I \subset B, |I| = \ell_1, \ldots, \ell_k \Big\}  \subset N(x,V).$$
\end{enumerate}
    
\end{theorem}



\begin{proof}
    The equivalence between \eqref{summary theorem pt 0} and \eqref{summary theorem pt 1} is proved in \cite{Host_Kra_Maass_nilstructure:2010}. Likewise, for \eqref{summary theorem pt 0} and \eqref{summary theorem pt 2} in \cite[Theorem 4.6]{radic2026Uniformity}. 
    It is clear that \eqref{summary theorem pt 4} implies \eqref{summary theorem pt 1}, \eqref{summary theorem pt 5} implies \eqref{summary theorem pt 2} and \eqref{summary theorem pt 6} implies all \eqref{summary theorem pt 1} to \eqref{summary theorem pt 5}.
    
    So we are only left to check that \eqref{summary theorem pt 0} implies \eqref{summary theorem pt 6}. For nilsystems this is a special case of \cref{cor contention of the image in the support primes}. For general pronilsystems, $Z = \lim_{\leftarrow} W_j$ with $(W_j, m_{W_j},T)$ nilsystems and $p_j \colon Z \to W_j$, given the neighborhood $V$ of $y$ there is a neighborhood $U \subset V$ of $y$ such that $U = p_j^{-1}(U')$ for some $U' \subset W_j$ neighborhood of $p_j(y)$. Using that $N(p_j(x),U') = N(x,U) \subset N(x,V) $ we conclude with the nilsystem case. 
\end{proof}

\section{Proof of the combinatorial results} \label{sec proofs}

In this section we derive the combinatorial results. 

\subsection{Infinite sumset in $U^k(\Phi)$-uniform sets}

We  use the following lemma.
\begin{lemma}[{\cite[Lemma 2.2, Corollary 4.6 and Theorem 4.9]{kmrr25}}]\label{EPs-implies-sumsets}
    Let $a\in X$ be a point in a measure-preserving system $(X,\mu,T)$ and $\ell \in \N$. Suppose that $E_1,\ldots,E_\ell \subset X$ are open sets and $\tau \in \mathcal{M}(X^{\ell+1})$ is progressive with $\tau(\{a \} \times X^\ell) =1$ and $\tau(X \times E_1 \times \cdots \times E_\ell) >0$. There exists an infinite set $B\subseteq \N$ such that 
    \begin{equation*}
        \Big\{ \sum_{b \in I} b \colon I \subset B, |I| = i \Big\}   \subset \{n\in \N \mid T^na\in E_i\} \quad \text{ for all } i = 1, \ldots, \ell.
    \end{equation*}
\end{lemma}

Now we prove one of the main results

\begin{proof}[Proof of \cref{main thrm uniform sets}]
    By \cref{furstenberg correspondence for sets}, there is an ergodic system $(X, \mu,T)$ with topological pronilfactors such that $a\in \gen(\mu,\Psi)$ and there are $U^k(X, \mu,T)$-uniform clopen sets $E_1, \ldots, E_k \subset X$ such that $T^na \in E_i$ if and only if $n \in A_i$ for all $i = 1, \ldots, k$. Take $\tilde \rho_a \in \cM( \mathbb{A}_\ell(X))$ as in \cref{def rho and sigma tilde} for $\ell = \ell_k$. By \cref{progressive}, $\tilde \rho_a$ is progressive so using  \cref{EPs-implies-sumsets}, we only need to prove that
    \begin{equation*}
        \int_{\A_\ell (X)} \prod_{i=1}^k \1_{E_i} ^{[\ell_i]} \diff \tilde\rho_a >0.
    \end{equation*}
    This inequality is a direct consequence of the good projection property (\cref{prop good projection properties}). Indeed, for $L=\{ \ell_1, \ldots, \ell_k\}$
    \begin{align*}
        \int_{\A_\ell (X)} \prod_{i=1}^k \1_{E_i} ^{[\ell_i]} \diff \tilde\rho_a& = \int_{Z_{\ell-1}^k} \bigotimes_{i=1}^k \E(\1_{E_i} \mid Z_{\ell-1}) \diff \rho_{\pi_{\ell-1}(a)}^L 
        \\
        &=\int_{Z^k_{k-1}} \bigotimes_{i=1}^k \E(\E(\1_{E_i} \mid Z_{\ell-1}) \mid Z_{k-1})   \diff \eta_{\pi_{k-1}(a)}^L .
    \end{align*}
     By nested conditional expectation, $\E(\E(\1_{E_i} \mid Z_{\ell-1}) \mid Z_{k-1})  = \E(\1_{E_i} \mid Z_{k-1}) \circ \pi_{k-1} = \mu(E_i) $ where in the last equality we are using the $U^k(X, \mu,T)$-uniformity. Thus,
    \begin{equation*}
        \int_{\A_\ell (X)} \prod_{i=1}^k \1_{E_i} ^{[\ell_i]} \diff \tilde\rho_a  = \prod_{i=1}^k \mu(E_i) >0
    \end{equation*}
    concluding the proof.
    \end{proof}

\subsection{Infinite prime sumsets}

\cref{uniformity-measures-in-primes} motivates then the following definition.
\begin{definition}
For a measure preserving system $(X,\mu,T)$ we say that a measure $\tau \in \cM(X^{\ell+1})$ is \emph{shifted prime-progressive} if for all open sets $E_1, \ldots, E_\ell \subset X$ with
\begin{equation*}
    \tau(X \times E_1 \times \cdots \times E_\ell) >0,
\end{equation*}
then there are infinitely many $n \in \P-1$ such that
\begin{equation*}
    \tau((X \times E_1 \times \cdots \times E_\ell ) \cap T_{\Delta}^{-n}(E_1 \times \cdots \times E_\ell \times X)) >0.
\end{equation*}

\end{definition}

We notice that the proof of \cref{EPs-implies-sumsets} generalizes trivially for shifted primes progressive measures, so we state it without proof. 
\begin{lemma}\label{prime-EPs-implies-sumsets}
        Let $a\in X$ be a point in a measure-preserving system $(X,\mu,T)$ and $\ell \in \N$. Suppose that $E_1,\ldots,E_\ell \subset X$ are open sets and $\tau \in \mathcal{M}(X^{\ell+1})$ is shifted primes-progressive with $\tau(\{a \} \times X^\ell) =1$ and $\tau(X \times E_1 \times \cdots \times E_\ell) >0$. There exists an infinite set $B\subseteq \P-1$ such that 
    \begin{equation*}
        \Big\{ \sum_{b \in I} b \colon I \subset B, |I| = i \Big\}   \subset \{n\in \N \mid T^na\in E_i\} \quad \text{ for all } i = 1, \ldots, \ell.
    \end{equation*}
\end{lemma}
 
Now we are in position to prove \cref{main thrm nilborh}.

\begin{proof}[Proof of \cref{main thrm nilborh}]
By the definition of $\nilbohr{}$ set, we have that there is an $k$-step nilsystem $(X,\mu,T)$, a point $a\in X$ and an open set $E\subseteq X$ such that
$$A\supset \{n\in \N \mid T^na \in E\}. $$
Consider $\kappa$ and $\rho_a^\circ$ as in \cref{prop kappa rho_x}. Using \cref{uniformity-measures-in-primes} and \cref{prime-EPs-implies-sumsets}, it is enough to prove  that there exists $t\in \N$ such that 
\begin{equation} \label{eq rho shifted positive}
    \rho_a^\circ(X\times T^{-t}E\times \cdots \times T^{-t}E)>0 .
\end{equation}
Notice that for any measurable set $U \subset X$, $\rho_a(X \times U \times X^{k-1}) = \mu(U)$, in particular $X$ has at most $\kappa$ connected components. With this let us define $\mu^\circ_a \in \cM(X)$ the Haar measure of the connected component of $X$ containing the point $a$ that we denote $X^\circ$. Since $0<\mu(E) = \frac{1}{\kappa} \sum_{s=0}^{\kappa-1}\mu^\circ_a (T^{-s}E) $, there is $0 \leq s < \kappa$ such that $\mu^\circ_a (T^{-s}E)>0$. Consider $\HP_k (X^\circ)$ the Hall-Petresco nilmanifold of the ergodic nilsystem $(X^\circ,\mu^\circ_a , T^\kappa)$.   Notice that for $x \in X^\circ$, $x \mapsto \rho^\circ_x$ coincides with the continuous $\tilde T ^\kappa$-ergodic decomposition of $m_{\HP_k (X^\circ)}$. 
Thus, by \cref{cor to get shift t}, there exists $t_0 \in \N$ such that
\begin{equation*}
    \rho_a^\circ(X\times T^{-(\kappa t_0 + s) }E\times \cdots \times T^{-(\kappa t_0 + s) }E)>\alpha/2
\end{equation*}
where $\alpha = m_{\HP_k (X^\circ)} (X\times T^{-s}E\times \cdots \times T^{-s}E)$ which is positive by Furstenberg multiple recurrence theorem \cite{Furstenberg77}. 
\end{proof}

We say that $A$ is a \textit{basic $\nilbohr{}$ set} if $A= \{ n \in \N \colon T^n x \in U\}$ for $(X,T)$ a minimal nilsystem, $U \subset X$ a non-empty open set and $x \in X$. Notice that, by definition, a $\nilbohr{}$ set is a set that contains a basic $\nilbohr{}$ set. 

\begin{remark} \label{remark single nilsystem}
    For a family of basic $\nilbohr{}$ sets $A_1, \ldots, A_k$, there is a single minimal nilsystem $(X,T)$ and point $x \in X$ such that $A_i= \{ n \in \N \colon T^n x \in V_i\}$ for open sets $V_i \subset X$. Indeed, consider $(X_i,T_i)$, $x_i \in X_i$ and $U_i \subset X_i$ from the definition of basic $\nilbohr{}$ set. Then we conclude by taking $x= (x_1, \ldots, x_k)$, $X = \overline{\{ (T_1 \times \cdots \times T_k)^n x \colon n \in \Z\}}$ and $V_i = \{ y \in X \colon y_i \in U_i\}$. 
\end{remark}

\begin{theorem}\label{main-theorem-2}
    Fix $k \geq 2$. Let $A_1, \ldots , A_k \subset \N$ be basic $\nilbohr{}$ sets and $U^k ( \Phi)$-uniform sets and let  $\ell_1, \ell_2, \ldots, \ell_k$ be distinct natural numbers. There exists $P \subset \P$ infinite such that
    \begin{equation}
        \Big\{ \sum_{p \in I} p \colon I \subset P, |I| = \ell_i \Big\}   \subset A_i \quad \text{ for all } i = 1, \ldots, k.
    \end{equation}
\end{theorem}

\begin{proof}
Since $A_i$ is a basic $\nilbohr{}$ set and $U^k ( \Phi)$-uniform set, the set $A_i - \ell_i$ is also a basic $\nilbohr{}$ set and $U^k ( \Phi)$-uniform set, for every $i =1, \ldots,k$.

Following steps similar to the proof of \cref{main thrm nilborh} and using \cref{remark single nilsystem}, there is an ergodic nilsystem $(X, \mu,T)$ such that $E_1, \ldots, E_k \subset X$ and $a \in X$ are open sets with $\{ n \in \N \colon T^n a \in E_i \} =  A_i - \ell_i $ and $E_i$ is $U^k(X, \mu,T)$-uniform, for all $i = 1, \ldots, k$.
Since $\E( \1_{E_i} \mid Z_{k-1})$ is constant $\mu$-almost surely and like for \cref{prop good projection properties}, we get that $\rho^\circ_a (X \times E_1 \times \cdots \times E_k) = \prod_{i=1}^k \mu(E_i)>0$. Using that $\rho^\circ_a$ is $(\P-1)$-progressive we get that there is $B \subset \P-1$ infinite such that
\begin{equation*}
    \Big\{ \sum_{b \in I} b \colon I \subset B, |I| = \ell_i \Big\} \subset \{ n \in \N \colon T^n a \in E_i \} \subset A_i - \ell_i \quad \text{ for all } i=1, \ldots, k.
\end{equation*}
Thus, taking $P= B+1$ we conclude the proof. 
\end{proof}

We finish proving \cref{main motivating thrm}.
\begin{proof}[{Proof of \cref{main motivating thrm}}]
Let $U\subseteq \T$ a non-empty interval and $Q$ a polynomial of degree $k \geq 2$ with irrational leading coefficient. By \cref{main-theorem-2}, it is enough to prove that the set 
    $$ A=\{n\in \N: Q(n)\in U\}$$
    is a basic $\nilbohr{}$ set and $U^k(\Phi)$-uniform for a F\o lner sequence $\Phi$. On one hand, the first statement comes from the classical construction in affine systems, see for example \cite[Proposition 3.11]{Furstenberg81}. On the other hand, if for $m \in \Z$, $\varphi_m(n) = \exp{(2\pi i m Q(n))}$, one can deduce by induction on the degree of $Q$ that $\norm{\varphi_m}_{U^k(\Phi)} =0$ for all F\o lner sequence $\Phi$ and all $m \neq 0$. We then deduce the uniformity of the set $A$ directly from \cite[Lemma 7.1]{radic2026Uniformity}.

    For the case $k=1$, $A= \{ n \in \N \colon n \alpha \in U\} $ is the return times of the point $a=0 \in \T$ under the irrational rotation by $\alpha$ $R$. Notice that $(\T,m, R)$ is a totally ergodic group rotation and in particular for all $\ell \in \N$ and $x \in X$, $\rho_x^\circ =\rho_x \in \cM(\HP_\ell(\T)) $. Therefore, again by total ergodicity, $\rho_a^\circ(X^\ell \times U) = \mu(U) >0$ concluding the proof.
\end{proof}

\footnotesize{
\bibliographystyle{abbrv}
\bibliography{refs}

}
\bigskip
\noindent
Felipe Hernández\\
\textsc{{\'E}cole Polytechnique F{\'e}d{\'e}rale de Lausanne} (EPFL)\par\nopagebreak
\noindent
\href{mailto:felipe.hernandezcastro@epfl.ch}
{\texttt{felipe.hernandezcastro@epfl.ch}}

\bigskip
\noindent
Trist\'an Radi\'c\\
\textsc{Northwestern University} \par\nopagebreak
\noindent
\href{tristan.radic@u.northwestern.edu}
{\texttt{tristan.radic@u.northwestern.edul}}

\end{document}